\documentclass[12pt,reqno]{amsart}
\usepackage[margin=1in]{geometry}
\makeatletter
\def\section{\@startsection{section}{1}%
	\z@{.7\linespacing\@plus\linespacing}{.5\linespacing}%
	{\bfseries
		\centering
}}
\def\@secnumfont{\bfseries}
\makeatother
\usepackage{amsmath,amssymb,amsthm,graphicx,amsxtra, setspace}
\usepackage[utf8]{inputenc}
\usepackage{charter}
\usepackage{mathrsfs}
\usepackage{alltt}
\usepackage{relsize}
\usepackage{hyperref}
\usepackage{aliascnt}
\usepackage{tikz}
\usepackage{mathtools}
\usepackage{multicol}
\usepackage{upgreek}
\usepackage{graphicx,type1cm,eso-pic,color}
\allowdisplaybreaks
\usepackage{scalerel,stackengine}
\stackMath
\newcommand\reallywidehat[1]{%
	\savestack{\tmpbox}{\stretchto{%
			\scaleto{%
				\scalerel*[\widthof{\ensuremath{#1}}]{\kern-.6pt\bigwedge\kern-.6pt}%
				{\rule[-\textheight/2]{1ex}{\textheight}}
			}{\textheight}%
		}{0.5ex}}%
	\stackon[1pt]{#1}{\tmpbox}%
}
\numberwithin{equation}{section}

\newtheorem{theorem}{Theorem}[section]

\newaliascnt{lemma}{theorem}
\newtheorem{lemma}[lemma]{Lemma}
\aliascntresetthe{lemma} 

\newaliascnt{proposition}{theorem}

\aliascntresetthe{proposition}

\newaliascnt{assumption}{theorem}

\aliascntresetthe{assumption}

\newaliascnt{corollary}{theorem}

\aliascntresetthe{corollary}

\newaliascnt{definition}{theorem}
\newtheorem{definition}[definition]{Definition}
\aliascntresetthe{definition}

\newaliascnt{example}{theorem}

\aliascntresetthe{example}

\newaliascnt{remark}{theorem}
\newtheorem{remark}[remark]{Remark}
\aliascntresetthe{remark}

\newaliascnt{hypothesis}{theorem}

\aliascntresetthe{hypothesis}

\newaliascnt{property}{theorem}

\aliascntresetthe{property}

\let\originalleft\left
\let\originalright\right
\renewcommand{\left}{\mathopen{}\mathclose\bgroup\originalleft}
\renewcommand{\right}{\aftergroup\egroup\originalright}

\makeatletter
\newcommand{\doublewidetilde}[1]{{%
		\mathpalette\double@widetilde{#1}%
}}
\newcommand{\double@widetilde}[2]{%
	\sbox\z@{$\m@th#1\widetilde{#2}$}%
	\ht\z@=.9\ht\z@
	\widetilde{\box\z@}%
}
\makeatother

\renewcommand{\d}{\/\mathrm{d}\/}

\def\w{\textbf{W}^{\varepsilon}_{{\theta}^{\varepsilon}}}

\def\L{\mathrm{L}}

\def\F{\mathrm{F}}
\def\C{\mathrm{C}}

\def\h{\mathbf{h}}
\def\J{\mathrm{J}}

\def\D{\mathrm{D}}
\def\y{\mathbf{y}}

\def\z{\mathbf{z}}
\def\v{\mathbf{v}}
\def\V{\mathbb{V}}
\def\w{\mathbf{w}}
\def\W{\mathrm{W}}
\def\G{\mathbb{G}}

\def\no{\nonumber}
\def\V{\mathbb{V}}

\def\U{\mathrm{U}}

\def\u{\mathbf{u}}
\def\H{\mathbb{H}}
\def\n{\mathbf{n}}

\newcommand{\R}{\mathbb{R}}

\renewcommand{\d}{\/\mathrm{d}\/}

\newcommand{\Addresses}{{
		\footnote{

			\noindent \textsuperscript{1}School of Mathematics, Indian Institute of Science Education and Research, Trivandrum (IISER-TVM),
			Maruthamala PO, Vithura, Thiruvananthapuram, Kerala, 695 551, INDIA  \par\nopagebreak \noindent
			\textit{e-mail:} \texttt{sheetal@iisertvm.ac.in}, *Corresponding Author

			\noindent \textsuperscript{2}School of Mathematics,
			Indian Institute of Science Education and Research, Trivandrum (IISER-TVM),
			Maruthamala PO, Vithura, Thiruvananthapuram, Kerala, 695 551, INDIA.  \par\nopagebreak \noindent
			\textit{e-mail:} \texttt{plnmn915@iisertvm.ac.in}

}}}

\begin{document}
\title{Nonlocal Cahn-Hilliard-Brinkman System with regular potential: Regularity and optimal control
\Addresses	}

       
	\author[S. Dharmatti and P. L. N. Mahendranath]
	{Sheetal Dharmatti\textsuperscript{1*} and P. L. N. Mahendranath\textsuperscript{2}}

\begin{abstract}
In this paper we study optimal control  problem for non local Cahn-Hilliard-Brinkman system which models  phase separation of binary fluids in porous media. We consider the system in two dimensional bounded domain with regular potential.  We extend  recently proved existence of weak solution results for such a system  and prove the existence of strong solution under certain assumptions on the forcing term and initial datum. Further using our regularity results, we study the tracking type optimal control problem. We prove the existence of an optimal control and establish the first order optimality condition. Lastly, we characterize optimal control in terms of the solution of corresponding adjoint system. The existence of solution for the adjoint system is also established. 
\end{abstract}



\maketitle
\section{Introduction} 
The Brinkman equation was proposed in \cite{Brinkman} by H. C. Brinkman. It is a modified Darcy's law to describe the flow through porous media. In recent studies, a diffuse interface variant of the Brinkman equation is proposed to model the phase separation of the incompressible binary fluids in a porous medium. The idea is to couple the Brinkman equation with the Cahn-Hilliard equation, which describes the phase separation phenomenon. Let $\Omega \subset \mathbb{R}^2$ be a bounded smooth domain with boundary $\partial \Omega$. Consider the Non local Cahn-Hilliard-Brinkman (CHB) system (see \cite{CHB weak}) given by,
\begin{align} \label{60}
\varphi_t + \nabla \cdot(\u \varphi) = \Delta \mu, &\quad \text{in} \quad \Omega \times (0,T),\\
\mu = a \varphi-\J * \varphi + \F'(\varphi), &\quad \text{in} \quad \Omega \times (0,T), \\
- \nabla \cdot (\nu (\varphi) \nabla \u) + \eta \u + \nabla \pi  = \mu \nabla \varphi + \h, &\quad \text{in} \quad \Omega \times (0,T),\\
\mathrm{div}(\u) =0,&\quad \text{in} \quad  \Omega \times (0,T),
\end{align}
We endow this system with following boundary and initial conditions,
\begin{align}
\frac{\partial \mu}{\partial \n} &=0, \ \mathrm{on} \ \partial \Omega,\\
\u &=0, \ \mathrm{on} \ \partial \Omega,\\
\varphi(0) &= \varphi_0, \ \mathrm{on} \  \Omega, \label{61}
\end{align}
where $\varphi$ denotes difference in concentrations of the two fluids, and $\u$ is the average fluid velocity. The viscosity coefficient, which may depend on $\varphi$ is denoted by $\nu >0$,  permeability is denoted by $\eta >0$ and $\pi$ is the pressure exerted on the fluid. Let $\J:\mathbb{R}^d \rightarrow \mathbb{R} $ be a suitable interaction kernel, $a(x) = \int_\Omega \J(x-y) \d y$ and the spatial convolution $\J*\varphi$ be defined by 
\begin{align*}
(\J*\varphi)(x):= \int_\Omega \J(x-y)\varphi (y) dy, \quad x \in \Omega.
\end{align*} 
The above system is called nonlocal because of the presence of the $\J$ term. The external forcing is denoted by $\h$, and $\F$ is a double-well potential accounting for phase separation, which can be singular (typically logarithmic potential) or regular (e.g., $\F (s) = (s^2-1)^2$). In this paper, we consider only the case of regular potential. If $\nu=0$, the system \eqref{60}-\eqref{61} becomes the so-called Cahn-Hilliard-Hele-Shaw system (or also referred to as Cahn-Hilliard-Darcy system in the context of a multi-phase fluid mixture in nonporous medium) (see \cite{CHHS sing}) and is used in modeling tumor growth dynamics. There is a surge of papers in recent years that study existence, uniqueness, numerics, and optimal control problems for Cahn-Hilliard-Darcy  system (see \cite{CHHS singnonlocal, CHDS num,CHD fem,on CHD,CHD op}). 

The local Cahn-Hilliard-Brinkman system is obtained by replacing $\mu$ equation in \eqref{60}-\eqref{61} by $ \mu = -\Delta \varphi + \F'(\varphi)$.  Well-posedness and some convergence results for the local Cahn-Hilliard-Brinkman system have been studied in \cite{CHB local} with regular potential. For the local system, the optimal control problem is studied in \cite{localBrinkman} and some numerical results  in \cite{CHB num}.

From the physical viewpoint, the nonlocal Cahn-Hilliard (CH) equation can be justified (see \cite{CH phy1,CH phy2,CH phy3}). The nonlocal CH equation has been analyzed theoretically and numerically in (e.g., \cite{CH Han, CH1,CH2, CH num}) under various assumptions on the potential $\F$. The nonlocal version of the Cahn Hilliard equation coupled with the Navier-Stokes system has been studied recently. For example results about existence of weak solution, strong solution, long time behaviour and optimal control problems are studied  in \cite{CHNS weak, CHNS strong, unique, CHNS traj, CHNS op}. 

The coupled nonlocal Cahn Hilliard Brinkman system  \eqref{60}-\eqref{61}  is studied recently, for existence and  uniqueness results  in \cite{CHB weak} in the case of dimension 2 and 3. We are interested in studying the optimal control problems related to the Cahn-Hilliard-Brinkman system. However, such results are not available in the literature as an optimal control problem requires a higher regularity of the solutions. In this work, we first address this issue. 
It has been challenging to prove regularity results for the nonlocal Cahn-Hilliard-Brinkman system. We prove the existence of strong solution for the Cahn-Hilliard-Brinkman system in two dimension. We follow the work of \cite{CHNS strong}  for the existence and uniqueness of strong solution results for the Cahn-Hilliard-Navier-Stokes system and results of \cite{CHB weak} to obtain these results.  We employ this regularity to further study the optimal control problem for the Cahn-Hilliard-Brinkman system. To the best of our knowledge, such a result is not available in the literature to date. Our aim is prove the existence of optimal control and finally to characterize it in terms of adjoint variables.

The structure of the paper is as follows. In the next section, we recall the existence results of the system \eqref{60}-\eqref{61} obtained in \cite{CHB weak}.  Further we  consider the system \eqref{60}-\eqref{61} with constant viscosity $\nu$ and $\eta=1$. We prove the existence of a  strong solution and obtain corresponding difference estimates. In section \ref{linear}, we prove three important results: the existence of the optimal control, the existence of a solution for the linearised system, and the differentiability of control to the state operator. In section \ref{adjoint} we derive the first-order necessary optimality condition. We further study the existence of a solution for the adjoint system. Finally, we characterize the optimal control in terms of the adjoint variable.

\section{ Existence of Strong Solution}
\subsection{Functional setting and Preliminary results}

We first explain the  functional spaces needed to obtain our main results. Let us define
\begin{align*}
\G_{\text{div}} &:= \left\{ \u \in L^2(\Omega;\R^n) : \text{div}(\u)=0,\  \u\cdot \mathbf{n} \big|_{\partial \Omega}=0   \right\}, \\
\V_{\text{div}} &:= \left\{\u \in H^1_0(\Omega;\R^n): \text{div}(\u)=0\right\},\\ 
H&:=L^2(\Omega;\R),\ V:=H^1(\Omega;\R),
\end{align*}
where $n=2,3$. Let us denote $\| \cdot \|$ and $(\cdot, \cdot)$, the norm and the scalar product, respectively, on both $\mathrm{H}$ and $\G_{\text{div}}$. The duality between any Hilbert space $X$ and its dual $X'$ is denoted by $_{X'}\langle \cdot,\cdot\rangle_{X}$. We know that $\V_{\text{div}}$ is endowed with the scalar product 
$$(\u,\v)_{\V_{\text{div}} }= (\nabla \u, \nabla \v)=2(\mathrm{D}\u,\mathrm{D}\v),\ \text{ for all }\ \u,\v\in\V_{\text{div}}.$$ The norm on $\V_{\text{div}}$ is given by $\|\u\|_{\V_{\text{div}}}^2:=\int_{\Omega}|\nabla\u(x)|^2\d x=\|\nabla\u\|^2$. For every $f \in V'$, we denote $\overline{f}$ the average of $f$ over $\Omega$, i.e., $\overline{f} := |\Omega|^{-1} {_{V'}}\langle f, 1 \rangle_{V}=|\Omega|^{-1}\int_{\Omega}f(x)\d x$. 
Let us also define the operator $\mathcal{B}$, an unbounded linear operator on $H$ with
domain $\D(\mathcal{B}) = \left\{v \in \mathrm{H}^2(\Omega) : \frac{\partial v}{\partial\mathbf{n}}= 0\text{ on }\partial\Omega \right\}$. 
%
We state below some estimates which will be used in this paper. 
\begin{lemma}[Sobolev inequality]
For $v \in H^{s}(\mathbb R^{n})$, we have
\[ \|v\|_{L^{q}(\mathbb R^{n})} \leq C_{n, s, q} \|v\|_{H^s(\mathbb R^{n})}\]
provided that q lies in the following range
\begin{enumerate}
\item[(i)] if $s < n/2$, then $2 \leq q \leq \frac{2n}{n-2s}.$
\item[(ii)] if $s = n/2$, then $2 \leq q < \infty$.
\item[(iii)] if $s > n/2$, then $2 \leq q \leq \infty$.
\end{enumerate}
\end{lemma}
\begin{lemma}\cite{nirenberg}
Let $u \in L^q(\mathbb{R}^n)$ and its derivatives of order $m$, $D^m u \in L^r(\mathbb{R}^n)$, $1 \leq q,r \leq \infty$. For the derivatives $D^j u, 0 \leq j <m$, the following inequalities hold
\begin{align}\label{gag}
\|D^j u\| _{L^p(\mathbb{R}^n)} \leq C \| D^m u\|^\theta_{L^r(\mathbb{R}^n)} \| u \|^{1-\theta}_{L^q(\mathbb{R}^n)}
\end{align}
where, 
\begin{align*}
\frac{1}{p} = \frac{j}{n} + \theta \left( \frac{1}{r}-\frac{m}{n} \right) + (1-\theta ) \frac{1}{q}
\end{align*}
for all $\theta$ in the interval $\frac{j}{m} \leq \theta \leq 1$. The constant $C$ depends only on $n,m,j,q,r,\theta$, with the following exceptional cases
\begin{enumerate}
\item If $j=0, \ rm < n, \ q = \infty$ then we make the additional assumption that either $u$ tends to zero at infinity or $u \in L^{\tilde{q}}$ for some finite $\tilde{q} >0 $
\item If $1 < r < \infty$, and $m-j-\frac{m}{r}$ is a non negative integer then \eqref{gag} holds only for a satisfying $\frac{j}{m} \leq \theta <1$.
\end{enumerate}
If we consider the smooth bounded domain $\Omega \subset \mathbb{R}^n$ then \eqref{gag} becomes 
\begin{align*} 
\|D^j u\| _{L^p(\Omega)} \leq C \| D^m u\|^\theta_{L^r(\Omega)} \| u \|^{1-\theta}_{L^q(\Omega)} + \|u\|_{L^s(\Omega)}
\end{align*}
for some $s>0$. In a particular case, for $p=\infty, m=2, j=0, n=2, r=q=2$ we get 
\begin{align}
\|u\|_{L^\infty} \leq C_\Omega \|u\|^{1/2} \|u\|^{1/2}_{H^2} \quad u \in H^2(\Omega).
\end{align}
\end{lemma}
\begin{lemma}[Ladyzhenskaya Inequality]\label{lady}
	For $\u\in\ \C_0^{\infty}(\Omega;\R^n), n = 2, 3$, there exists a constant $C$ such that
	\begin{equation}
	\|\u\|_{\mathbb{L}^4}\leq C^{1/4}\|\u\|^{1-\frac{n}{4}}\|\nabla\u\|^{\frac{n}{4}},\text{ for } n=2,3,
	\end{equation}
	where $C=2,4$ for $n=2,3$ respectively. Using Poincaré's inequality we can deduce that for $n=2,3$,
	\begin{align}\label{poin}
	\|\u\|_{\mathbb{L}^4} \leq C_\Omega \|\nabla \u\| .
	\end{align}
\end{lemma}

We need following assumptions to deduce well-posedness of Cahn Hilliard Brinkman system under consideration:

\begin{itemize}
\item[(\textbf{H1})] ~~ Let $\Omega \subset \mathbb{R}^d, d = 2,3,$ be an open, bounded and connected  domain with a smooth boundary. 
\item[(\textbf{H2})] ~~ $\J \in W^{1,1} (\mathbb{R}^d)$ satisfies $ \J(x) = \J(-x),$ and
\begin{align*}
a(x) := \int_\Omega \J(x-y) \d x \geq 0 , \ a.e. \ x \in \Omega.
\end{align*}
\item[(\textbf{H3})] ~~$\F \in C^{2,1}_{loc} (\mathbb{R})$ and there exists $c_0 > 0$ such that
\begin{align}\label{H3}
\F'' (S) + a(x) \geq c_0 , \quad \forall s \in \mathbb{R}, \ a.e. \ x \in \Omega.
\end{align}
\item[(\textbf{H4})] ~~ There exist $c_1>0,c_2>0$ and $q>0$ if $d=2, q \geq \frac{1}{2}$ if $d=3$ such that 
\begin{align*}
\F''(s) + a(x) \geq c_1 |s|^{2q} - c_2\quad \forall s \in \mathbb{R}, \ a.e. \ x \in \Omega.
\end{align*}
\item[(\textbf{H5})] ~~ There exist $c_3 >0$ and $p \in (1,2]$ such that 
\begin{align*}
|\F'(s)|^p \leq c_3 (|\F(s)+ 1|) \quad \forall s \in \mathbb{R}.
\end{align*}
\item[(\textbf{H6})] ~~ $\nu$ is Lipschitz on $\mathbb{R}$ and there exist $\nu_0, \nu_1 >0$ such that 
\begin{align*}
\nu_0 \leq \nu (s) \leq \nu_1, \quad \forall  s \in \mathbb{R},
\end{align*}
and $\eta \in L^\infty(\Omega)$ is such that $\eta(x) \geq 0, \ \mathrm{a.e.} \ x \in \Omega$.
\end{itemize}	
Now, we summarize few results from \cite{CHB weak} regarding well-posedness and the uniqueness of the system: 		
	\begin{definition}
	Let $T>0$ be given and let $\varphi_0 \in H $ be such that $F(\varphi_0) \in L^1(\Omega)$. Then $(\varphi, \u)$ is a weak solution of \eqref{60}-\eqref{61} if 
	\begin{align*}
	&\varphi \in C([0,T];H) \cap L^2(0,T;V) \\
	&\varphi_t \in L^2(0,T;V') \\
	&\mu = a \varphi  - \J * \varphi +\F'(\varphi)\in L^2(0,T;V) \\
	& \u \in L^2(0,T;\V_{\mathrm{div}} )
	\end{align*}
	and it satisfies 
	\begin{align*}
	& \langle \varphi_t, \psi \rangle + (\nabla \mu , \nabla \psi) = (u \varphi, \nabla \psi), \quad \forall \psi \in V, \quad a.e. \ \mathrm{ in } \ (0,T), \\
	& (\nu(\varphi)\nabla \u, \nabla \v) + (\eta\u, \v)= (\mu \nabla \varphi, \v) + \langle \h, \v \rangle, \quad \forall \v \in \G_{\mathrm{div}}, \quad a.e. \ \mathrm{ in } \ (0,T), \\
	& \varphi(0) = \varphi_0, \quad a.e. \ \mathrm{in} \ \Omega.
	\end{align*}
	\end{definition}
	\begin{theorem} \label{existence} [\cite{CHB weak}, Theorem 2.2]
	Suppose that (\textbf{H1})-(\textbf{H6}) are satisfied. Let $\varphi_0 \in H $ be such that $\F (\varphi_0) \in L^1 (\Omega)$ and $\h \in L^2(0,T; \V'_{\mathrm{div}})$. Then there exists a weak solution $(\varphi, \u)$ of \eqref{60}-\eqref{61}. Furthermore, $F(\varphi)$ is in $ L^{\infty} (0,T; L^1(\Omega))$ and setting
	\begin{align*}
	\mathcal{E}(\varphi (t)) = \frac{1}{4} \int_\Omega \int_\Omega \J (x-y) (\varphi(x)- \varphi(y))^2 \d x \d y + \int_\Omega \F(\varphi(x)) \d x
	\end{align*}
	the following energy equality holds for almost every $t \in (0,T)$
	\begin{align*}
	\frac{d}{dt} \mathcal{E}(\varphi(t)) + \|\nabla \mu \|^2+ \nu \|\nabla \u \|^2 + \| \u \| = \langle \h, \u \rangle.
	\end{align*}
	\end{theorem}

	\begin{theorem}[\cite{CHB weak}, Proposition 2.1]
 Let assumptions of Theorem \ref{existence} hold. If $\varphi_0 \in L^\infty(\Omega)$ then any solution $(\varphi, \u)$ to the problem on $[0,T]$ corresponding to $\varphi_0$ satisfies 
      \begin{align*}
      \varphi, \mu \in L^\infty (\Omega \times (0,T)).
      \end{align*}
	\end{theorem}
	
\begin{theorem} \label{regularity} [\cite{CHB weak}, Corollary 2.1]
 Let (\textbf{H1})-(\textbf{H6}) hold. If $\h \in L^\infty(0,T; \V_{\text{div}}')$ for some $T>0$. Then any weak solution $(\varphi, \u)$ to \eqref{60}-\eqref{61} is such that
\begin{align*}
\varphi \in L^4(0,T;L^4(\Omega)), \quad \u \in L^\infty(0,T;\V_{\text{div}}).
\end{align*}
\end{theorem}	
The following result can be proved using [\cite{CHB weak}, Proposition 2.2].
	\begin{theorem} \label{weakunique}
	Let hypotheses (\textbf{H1})-(\textbf{H6}) hold. Suppose 	$\h_1, \h_2 \in L^\infty(0,T;\V_{\mathrm{div}}')$. Consider two weak solutions to \eqref{60}-\eqref{61}, namely $(\varphi_1, \u_1)$ and $(\varphi_2, \u_2)$, corresponding to the initial data $\varphi_{1,0}$ and $\varphi_{2,0}$ such that $\varphi_{i,0} \in L^2(\Omega)$ and $F(\varphi_{i,0}) \in L^1(\Omega), i=1,2$. Then there exists $N=N(T)>0$ such that, for any $t \in [0,T]$, 
	\begin{align}
	&\| \varphi_1(t) - \varphi_2(t) \|^2_{V'}   + \int_0^t \|\u_1(t) - \u_2(t) \|^2_{\V_{\mathrm{div}}} \no \\
	& \hspace{1cm}\leq N (\|\varphi_{1,0}- \varphi_{2,0}\|^2_{V'} + |\bar{\varphi}_{1,0} - \bar{\varphi}_{2,0}|)\|\h_1-\h_2\|_{L^2(0,T;\V_{\mathrm{div}}')} \label{uniq}
	\end{align}
	In particular, \eqref{16}-\eqref{op2} has a unique solution.
	\end{theorem}
	\begin{proof}
	Let us denote $\varphi = \varphi_1 - \varphi_2, \u = \u_1-\u_2$ and $\h = \h_1-\h_2$. Arguing exactly as in the proof of Proposition 2.2 in \cite{CHB weak} we can arrive at \eqref{uniq}. 
	\end{proof}

\subsection{ Strong Solution}
Let us consider the Cahn-Hilliard-Brinkman system
	\begin{align}
	\varphi_t + \u \cdot \nabla \varphi & = \Delta \mu, \quad \text{in} \quad \Omega \times (0,T), \label{16}\\
	\mu &= a \varphi  - \J * \varphi +\F'(\varphi),   \quad \text{in} \quad \Omega \times (0,T),\\
	-\nu \Delta \u + \u + \nabla \pi &= \mu \nabla \varphi + \h, \quad \text{in} \quad \Omega \times (0,T), \label{op1} \\
	\text{div} \, (\u) & = 0 , \quad \text{in} \quad \Omega \times (0,T),\\
		\u = \frac{\partial \mu}{\partial \n}&  = 0 , \quad \text{on} \quad \partial\Omega \times (0,T),\\
	\varphi (0)  & = \varphi_0 (x), \quad \text{in} \quad \Omega. \label{op2}
	\end{align}
which is obtained by assuming $\eta =1 $ and $\nu $ is independent of $\varphi$. We are now going to prove the main theorem of this section, namely the existence of a strong solution of the system \eqref{16}-\eqref{op2} for the dimensions $d=2$. We consider the space 
\begin{align*}
\mathcal{U} := \{ \h \in L^\infty (0,T; \G_{\mathrm{div}}) \quad | \quad \h_t \in L^2(0,T; \V'_{\mathrm{div}})\}.
\end{align*}
We observe that $\mathcal{U}$ is a Banach space with the norm (see Chapter 7 in \cite{Tomas})
\begin{align*}
\|\h\|_{\mathcal{U}} := \|\h\|_{L^\infty (0,T; \G_{\mathrm{div}})} + \|\h_t\|_{L^2(0,T; \V'_{\mathrm{div}})}.
\end{align*}
\begin{theorem} \label{strongsol}
Let $\h \in \mathcal{U}$ and $\varphi_0 \in H^2(\Omega) \cap L^\infty (\Omega)$ and hypothesis (\textbf{H1})-(\textbf{H5}) are satisfied. 
Further assume that $\F \in C^3(\mathbb{R}), \J \in W^{2,1}(\mathbb{R}^d)$. Then there exists a unique strong solution for the system \eqref{16}-\eqref{op2} on $[0,T]$ in the following sense 
\begin{align}
&\varphi \in  L^\infty(0,T;H^2(\Omega)), \label{43}\\
& \varphi_t \in L^\infty (0,T;H) \cap L^2(0,T; V) , \label{45} \\
&\u \in L^2(0,T; \H^2). \label{42} 
\end{align}
\end{theorem}
\begin{proof}
Note that by Theorem \ref{existence}, \ref{regularity} and \ref{weakunique} there exists a unique weak solution of \eqref{16}-\eqref{op2} under given assumptions. To prove higher regularity given by \eqref{43} we take inner product of \eqref{op1} with $- \Delta \u $. We get 
	\begin{align} \label{9}
	\nu \| \Delta \u \|^2 + \| \nabla \u \|^2 - (\nabla \pi ,\Delta \u)&= -(\mu \nabla \varphi, \Delta \u )- (\h,\Delta \u).
	\end{align}
	Since $\textrm{div}(\u) =0$, we have $-(\nabla \pi, \Delta \u) =( \pi, \Delta (\textrm{div}(\u))) =0$. Henceforth we shall denote by a positive constant $C = C(\J,\F, \Omega, \nu)$ and $C$ may vary from line to line, even within the estimate. Before we estimate the right hand side of \eqref{9} observe that
	\begin{align*}
	\mu \nabla \varphi &= (a \varphi - \J * \varphi + \F'(\varphi) ) \nabla \varphi \\
	& = \nabla \left(\F (\varphi) + a \frac{\varphi^2}{2}\right) -\nabla a \frac{\varphi^2}{2} - (\J * \varphi) \nabla \varphi.
	\end{align*}
	Hence we have
	\begin{align*}
	-(\mu \nabla \varphi, \Delta \u) = (\nabla a \frac{\varphi^2}{2}, \Delta \u) + ((\J * \varphi) \nabla \varphi ,\Delta \u).
	\end{align*}
Then, using integration by parts, H\"older's inequality and Young's inequality we get 
\begin{align} 
|(\mu \nabla \varphi, \Delta \u)| 
& \leq \frac{1}{2}\|\nabla a \|_{L^\infty} \|\varphi^2\| \| \Delta \u\| + \|\nabla \J * \varphi \|_{L^4} \|\varphi \|_{L^4} \|\Delta \u \| \no \\
& \leq \frac{1}{2}\| \nabla a \|_{L^\infty} \| \varphi \|^2_{L^4} \| \Delta \u\| + \| \nabla \J \|_{L^1} \| \varphi\|_{L^4}^2  \| \Delta \u \|  \label{136} \\
& \leq \frac{\nu}{3} \| \Delta \u \|^2 + C \| \varphi \|_{L^4}^4 ,\label{10}
\end{align}
and
\begin{align}\label{11}
|(\h,\Delta \u) | \leq \| \h \| \| \Delta \u \|
 \leq \frac{\nu}{6} \| \Delta \u \|^2 + C \| \h \|^2 .
\end{align}
Substitute \eqref{10} and \eqref{11} in \eqref{9}, we get, 
\begin{align} \label{87}
\frac{\nu}{2} \| \Delta \u \|^2 \leq C (\| \varphi \|_{L^4}^4 + \| \h \|^2). 
\end{align}
Integrate \eqref{87} from $0$ to $T$ and using Theorem \ref{regularity}, we get 
\begin{align*}
\frac{\nu}{2} \int_0^T \| \Delta \u (t) \|^2 \d t &\leq C \left( \int_0^T  \| \varphi (t)\|_{L^4}^4 \d t +\frac{3}{2 \nu} \int_0^T \| \h(t) \|^2 \d t \right) < \infty.
\end{align*}	
This proves \eqref{42}. Our next aim is to prove $\varphi \in L^\infty (0,T;H^2 (\Omega))$.
For, consider 
\begin{align}
(\Delta \mu , \Delta \varphi) &= (\Delta (a \varphi  - \J * \varphi +\F'(\varphi)) , \Delta \varphi) \no \\
&= ((a+\F''(\varphi) ) \Delta \varphi +\varphi \Delta a + 2 \nabla a \cdot \nabla \varphi - \Delta \J * \varphi + \F'''(\varphi) (\nabla \varphi)^2 , \Delta \varphi) \no\\
&\geq C_0 \| \Delta \varphi \|^2 - \frac{C_0}{16} \| \Delta \varphi\|^2 - C \|\varphi \|^2 \|\Delta a \|^2_{L^\infty} - \frac{C_0}{16} \| \Delta \varphi \|^2 - C \|\nabla a \|^2 \|\nabla \varphi \|^2 \no\\
& \ \ -\frac{C_0}{16} \| \Delta \varphi \|^2 - C \|\Delta \J\|^2 \| \varphi \|^2 -\frac{C_0}{16} \| \Delta \varphi \|^2 - C \| \nabla \varphi \|^4_{L^4} \no \\
&\geq \frac{3C_0}{4} \| \Delta \varphi \|^2 -C (\|\varphi \|_{\mathrm{V}}^2+\| \nabla \varphi \|^4_{L^4)}  .\label{14}
\end{align}
and
\begin{align}\label{15}
|(\Delta \mu, \Delta \varphi)| \leq \| \Delta \mu \| \| \Delta \varphi\| \leq \frac{C_0}{4} \| \Delta \varphi\|^2 + C  \| \Delta \mu \|^2.
\end{align}
Using \eqref{14} and \eqref{15} we get,
\begin{align} \label{140}
\frac{C_0}{2} \| \Delta \varphi \|^2 \leq C (\| \Delta \mu \|^2 +\|\varphi \|_{\mathrm{V}}^2+\| \nabla \varphi \|^4_{L^4}).
\end{align} 
Observe that we can find bound for the RHS by finding the bounds for the terms $\|\Delta \mu\| $ and $\|\nabla \varphi\|_{L^4} $.
Now we prove  that $  \Delta \mu \in L^\infty(0,T;H)$ and $\nabla \varphi \in L^\infty (0,T; L^4(\Omega))$. In fact we prove that $\nabla \varphi \in L^\infty (0,T; L^p(\Omega))$ for $2 \leq p \leq \infty$. From \eqref{16}, using H\"older inequality and \eqref{poin} we have
\begin{align}
\| \Delta \mu \| &\leq \| \varphi_t \| + \| \u \cdot \nabla \varphi \|, \no \\
&\leq \| \varphi_t \| + \| \u \|_{L^4} \|\nabla \varphi \|_{L^4} ,\no  \\
&\leq \| \varphi_t \| + C\|\nabla \u \| \|\nabla \varphi \|_{L^4} . \label{143}
\end{align} 
Since we have from Theorem \ref{regularity} that $\u \in L^\infty(0,T; \V_{\text{div}})$, we only have to prove that $\varphi_t \in L^\infty(0,T; H)$.
Before that we prove that $\varphi_t \in \L^2(0,T;H)$. Taking inner product of \eqref{16} with $\mu_t$, we get
\begin{align}
0 &= (\varphi_t, \mu_t) + (\u \cdot \nabla \varphi,\mu_t)  + \frac{1}{2} \frac{\d}{\d t} \| \nabla \mu \|^2 ,\no \\
&= ( (a +\F''(\varphi) )\varphi_t,\varphi_t)-(\varphi_t,\J * \varphi_t)+ (\u \cdot \nabla \varphi, \mu_t) + \frac{1}{2} \frac{\d}{\d t} \| \nabla \mu \|^2 \label{17}.
\end{align}
Now,
\begin{align}
|(\u \cdot \nabla \varphi, \mu_t)| &= |(\u \cdot \nabla \varphi, a \varphi_t  - \J * \varphi_t +\F''(\varphi) \varphi_t)| \no \\
&\leq \| a\|_{L^\infty} \| \varphi_t\| \|\u \|_{\mathbb{L}^\infty} \|\nabla \varphi \| + \| \J\|_{L^1} \| \varphi_t\| \|\u \|_{\mathrm{L}^\infty} \|\nabla \varphi \|+ C \| \varphi_t\| \|\u \|_{\mathbb{L}^\infty} \|\nabla \varphi \| \no \\
&\leq \frac{C_0}{12} \| \varphi_t\|^2 + C\|\u \|^2_{\H^2} \|\nabla \varphi \|^2 + \frac{C_0}{12} \| \varphi_t\|^2+ \|\u \|^2_{\H^2} \|\nabla \varphi \|^2 \no \\
& \qquad \qquad +\frac{C_0}{12} \| \varphi_t\|^2 + C \|\u \|^2_{\H^2} \|\nabla \varphi \|^2  \no \\
& \leq \frac{C_0}{4}\| \varphi_t\|^2 + C \|\u \|^2_{\H^2} \|\nabla \varphi \|^2, \label{18}
\end{align}
and 
\begin{align}
(\varphi_t,\J * \varphi_t) &= (-\u \cdot \nabla \varphi + \Delta \mu ,\J * \varphi_t)\no \\ 
& \leq |(\u \cdot \nabla \varphi, \J * \varphi_t )|+ |(\nabla \mu, \nabla \J * \varphi_t)| \no \\
& \leq \| \u \|_{L^\infty} \| \nabla \varphi\| \| \J\|_{L^1} \| \varphi_t \| + \| \nabla \mu\| \| \nabla \J * \varphi_t\| \no \\
& \leq \frac{C_0}{8} \| \varphi_t\|^2 + C\| \u \|^2_{\H^2} \| \nabla \varphi\|^2 + \frac{C_0}{8} \| \varphi_t\|^2 + C \| \nabla \mu\|^2 \no \\
& \leq \frac{C_0}{4} \| \varphi_t\|^2 + C (\| \u \|^2_{\H^2} \| \nabla \varphi\|^2 + \| \nabla \mu\|^2). \label{19}
\end{align}
Substituting \eqref{18} and \eqref{19} in \eqref{17} and using (\textbf{H3}) we get, 
\begin{align*}
\frac{1}{2} \frac{\d}{\d t} \| \nabla \mu \|^2 + \frac{C_0}{2} \|\varphi_t \|^2 \leq C (\| \u \|^2_{\H^2} \| \nabla \varphi\|^2 +\| \nabla \mu\|^2).
\end{align*}
Integrating from $0$ to $t$, we get,
\begin{align}
\frac{1}{2} \| \nabla \mu(t) \|^2 &+ \frac{C_0}{2}\int_0^t \|\varphi_s (s)\|^2 \d s \no \\
&\leq \frac{1}{2} \| \nabla \mu_0 \|^2 + C \int_0^t \| \u(s) \|^2_{\H^2} \| \nabla \varphi(s)\|^2 \d s+C \int_0^t \| \nabla \mu(s)\|^2 \d s .\label{23}
\end{align} 
Observe that,
\begin{align}
(\nabla \mu, \nabla \varphi)& = (a \nabla \varphi + \nabla a \varphi - \nabla \J * \varphi + \F''(\varphi) \nabla \varphi) \no \\
& \geq C_0 \| \nabla \varphi\|^2 - 2 \|\nabla \J \|_{L^1} \| \nabla \varphi\| \| \varphi\| \no \\
& \geq \frac{C_0}{2} \| \nabla \varphi \|^2 - C \| \nabla \varphi\|^2, \label{20}
\end{align}
and 
\begin{align}
(\nabla \mu, \nabla \varphi) \leq \frac{C_0}{4} \| \nabla \varphi \|^2 + C \| \nabla \mu \|^2. \label{21}
\end{align}
From \eqref{20} and \eqref{21} we get,
\begin{align}
\| \nabla \mu \|^2 \geq \frac{C_0}{4} \| \nabla \varphi\|^2 - C \| \varphi \|^2. \label{22}
\end{align}
Using \eqref{22}, from \eqref{23} we get
\begin{align*}
\frac{1}{2} \| \nabla \mu(t) \|^2 \leq \frac{1}{2} \| \nabla \mu_0 \|^2 + C\int_0^t (\| \u(s) \|^2_{\H^2}+1) \| \nabla \mu(s) \|^2 \d s + C \int_0^t \| \varphi(s) \|^2 \d s.
\end{align*}
From Gronwall's lemma we have
\begin{align} \label{41}
\nabla \mu \in L^\infty (0,T;H), \quad \forall \, T>0.
\end{align}
So from \eqref{22}, we conclude 
\begin{align} \label{137}
\varphi \in L^\infty (0,T; V).
\end{align}
Moreover from \eqref{23}, we get that 
\begin{align} \label{36}
\varphi_t \in L^2(0,T;H), \quad \forall \, T>0.
\end{align}
Now we prove that $ \varphi_t \in \L^\infty(0,T;H) $. Differentiate \eqref{16} and \eqref{op1} with respect to $t$ and take $L^2$ inner product with $\mu_t$ and $\u_t$ respectively,
\begin{align}
(\varphi_{tt}, \mu_t) + (\u_t \cdot \nabla \varphi, \mu_t)+(\u \cdot \nabla \varphi_t, \mu_t) &= - \|\nabla \mu_t\|^2, \label{33} \\
\nu \|\nabla \u_t\|^2 + \|\u_t\|^2 + (\nabla \pi_t, \u_t) &= (\mu_t \nabla \varphi, \u_t)+(\mu \nabla \varphi_t, \u_t) + (\h_t, \u_t) .\label{32}
\end{align}
Adding \eqref{32} and \eqref{33}, using the fact that $(\nabla \pi_t, \u_t)= (\pi_t, (\text{div}\, \u)_t) =0$, we get that
\begin{align}
(\varphi_{tt}, \mu_t)+ \nu \|\nabla \u_t\|^2 + \|\u_t\|^2 +\|\nabla \mu_t\|^2+(\u \cdot \nabla \varphi_t, \mu_t) = (\mu \nabla \varphi_t, \u_t) + \langle\h_t, \u_t \rangle. \label{34}
\end{align}
Now observe that
\begin{align}
(\varphi_{tt} ,\mu_t )&=  (\varphi_{tt},a \varphi_t - \J*\varphi_t + \F''(\varphi) \varphi_t) \no \\
&= \left(\frac{1}{2} \frac{\d}{\d t} \int_\Omega a \varphi_t^2 \right) -( - \u_t \cdot \nabla \varphi - \u \cdot \nabla \varphi_t + \Delta \mu_t,\J*\varphi_t) + \int_\Omega \F''(\varphi) \varphi_t \varphi_{tt} \no \\
&= \frac{1}{2} \frac{\d}{\d t} \left(\int_\Omega (a + \F''(\varphi))\varphi_t^2 \right) +( \u_t \cdot \nabla \varphi ,\J*\varphi_t)+( \u \cdot \nabla \varphi_t,\J*\varphi_t) \no \\
& \hspace{5cm} - (\Delta \mu_t,\J*\varphi_t) + \frac{1}{2}\int_\Omega \F'''(\varphi) \varphi_t^3 .\label{26}
\end{align}
Using \eqref{26} in \eqref{34}, we get 
\begin{align}
&\frac{1}{2} \frac{\d}{\d t} \left(\int_\Omega (a + \F''(\varphi))\varphi_t^2 \right) +  \nu \|\nabla \u_t\|^2 + \|\u_t\|^2 +\|\nabla \mu_t\|^2 \no \\
& \quad= (\mu \nabla \varphi_t, \u_t)-(\u \cdot \nabla \varphi_t, \mu_t) + \langle \h_t, \u_t\rangle -( \u_t \cdot \nabla \varphi ,\J*\varphi_t)-( \u \cdot \nabla \varphi_t,\J*\varphi_t) \no \\
& \hspace{2cm}+(\Delta \mu_t,\J*\varphi_t)-\frac{1}{2}\int_\Omega \F'''(\varphi) \varphi_t^3. \label{142}
\end{align}
Before we estimate right hand side terms of the above equality, we estimate $\nabla \varphi$ with $\nabla \mu$ in $L^p$ for $2 \leq p < \infty$ as follows
\begin{align*}
\int_\Omega \nabla \varphi |\nabla \varphi|^{p-2} \cdot \nabla \mu &= \int_\Omega \nabla \varphi |\nabla \varphi|^{p-2} \cdot (a \nabla \varphi + \varphi \nabla a - \nabla \J*\varphi + \F''(\varphi) \nabla \varphi) \\
&= \int_\Omega (a + \F''(\varphi)) |\nabla \varphi|^p + \int_\Omega (\varphi \nabla a - \nabla \J * \varphi) |\nabla \varphi|^{p-1}.
\end{align*}
Using \eqref{H3}, we get 
\begin{align*}
C_0 \|\nabla \varphi\|_{L^p}^p &\leq  \int_\Omega \nabla \varphi |\nabla \varphi|^{p-2} \cdot \nabla \mu - \int_\Omega \varphi \nabla a|\nabla \varphi|^{p-1} + \int_\Omega (\nabla \J * \varphi) |\nabla \varphi|^{p-1}\\
& \leq \| \nabla \mu \|_{L^p} \| \nabla \varphi \|_{L^p}^{p-1} + (\| \varphi \|_{L^p} \| \nabla a \|_{L^\infty} + \| \nabla \J \|_{L^1} \|\varphi\|_{L^p}) \| \nabla \varphi\|_{L^p}^{p-1}\\
& \leq \frac{C_0}{4} \| \nabla \varphi \|_{L^p}^p + C \| \nabla \mu \|_{L^p}^p+\frac{C_0}{4} \| \nabla \varphi \|_{L^p}^p + C \|\varphi\|_{L^p}^p.
\end{align*}
Therefore, we get
\begin{align} \label{27}
\|\nabla \varphi\|_{L^p} \leq C \| \nabla \mu \|_{L^p} + C.
\end{align}
where $C$  depends on $p$. Now we estimate $\| \nabla \mu \|_{L^p}$ in terms of $\| \varphi_t \|$. Using Gagliardo-Nirenberg inequality we get 
\begin{align*}
\|\nabla \mu \|_{L^p} &\leq C \| \nabla \mu \|^{2/p} \|\nabla \mu \|_{H^1}^{1-2/p}\\
& \leq C\| \nabla \mu \|^{2/p} \| \mu \|_{H^2}^{1-2/p}.
\end{align*}
Now using the fact that $\| \mu \|_{H^2} \cong \|\Delta \mu + \mu \| $ we infer that,
\begin{align*}
\|\nabla \mu \|_{L^p} &\leq C \| \nabla \mu \|^{2/p} (\| \Delta \mu \|^{1-2/p}+ \|\mu\|^{1-2/p} )\\
& \leq C \| \nabla \mu \|^{2/p} ( (\| \varphi_t\|^{1-2/p} + \| \u \cdot \nabla \varphi \|^{1-2/p})+ \|\mu\|^{1-2/p} ). \\
\end{align*}
From \eqref{41} and H\"older inequality we get
\begin{align*}
\|\nabla \mu \|_{L^p} \leq C (\| \varphi_t \|^{1-2/p} + \|\u\|_{L^q}^{1-2/p} \|\nabla \varphi \|_{L^p}^{1-2/p} + 1).
\end{align*}
where $\frac{1}{2}=\frac{1}{p}+\frac{1}{q}$.
By using \eqref{27} and Holder inequality we get
\begin{align*}
\|\nabla \mu \|_{L^p} \leq \frac{1}{2}\|\nabla \mu \|_{L^p}  +C (\| \varphi_t \|^{1-2/p} + \|\u\|_{L^q}^{p/2} +1).
\end{align*}
Finally we get 
\begin{align} \label{37}
\| \nabla \mu \|_{L^p} \leq C ( 1 + \| \varphi_t\|^{1-2/p}).
\end{align}
Using \eqref{37} we estimate $\nabla \varphi_t$ by $\nabla \mu_t$. For, consider 
\begin{align}
(\nabla \mu_t ,\nabla \varphi_t) &= (a \nabla \varphi_t + \varphi_t \nabla a - \nabla \J *\varphi_t + \F''(\varphi) \nabla \varphi_t + \F'''(\varphi) \varphi_t \nabla \varphi , \nabla \varphi_t) \no \\
& \geq C_0 \|\nabla \varphi_t\|^2+(\nabla a \varphi_t - \nabla \J * \varphi_t ,\nabla \varphi_t) + \int_\Omega \F'''(\varphi) \varphi_t \nabla \varphi \cdot \nabla \varphi_t. \label{28}
\end{align}
 The last term of the above equation can be estimated  using Gagliardo-Nirenberg inequality and from \eqref{27} as,
\begin{align}
\int_\Omega \F'''(\varphi) \varphi_t \nabla \varphi \cdot \nabla \varphi_t &\leq C \| \varphi_t \|_{L^6} \| \nabla \varphi \|_{L^3} \|\nabla \varphi_t \| \no \\
&\leq C (\| \nabla \varphi_t \|^{2/3} \|\varphi_t\|^{1/3} + \|\varphi_t\|)(1 + \|\varphi_t\|^{1/3}) \|\nabla \varphi_t \| \no \\
&\leq C  (\| \nabla \varphi_t \|^{5/3} \| \varphi_t \|^{1/3} + \| \nabla \varphi_t \|^{5/3} \| \varphi_t \|^{2/3}+ \| \nabla \varphi_t \| \| \varphi_t \| \no \\
& \hspace{2cm}+\| \nabla \varphi_t \|\| \varphi_t \|^{4/3})\no\\
&\leq \frac{C_0}{16} \| \nabla \varphi_t \|^2 + C \| \varphi_t\|^2 + \frac{C_0}{16} \| \nabla \varphi_t \|^2 + C \| \varphi_t \|^4 \no \\
& \quad + \frac{C_0}{16} \| \nabla \varphi_t \|^2 + C \| \varphi_t \|^2 + \frac{C_0}{16} \| \nabla \varphi_t \|^2 + C \| \varphi_t \|^{8/3}\no\\ 
&\leq \frac{C_0}{4} \| \nabla \varphi_t \|^2 + C \| \varphi_t\|^2 + C \| \varphi_t \|^4 + C . \label{29}
\end{align}
Using \eqref{H3}, from \eqref{28} and \eqref{29}, we get that
\begin{align}
(\nabla \mu_t ,\nabla \varphi_t) &\geq C_0 \| \nabla \varphi_t \|^2 -\frac{C_0}{8} \| \nabla \varphi_t \|^2 - C \| \nabla a \|_{L^\infty}\|\varphi_t \|^2 -\frac{C_0}{8} \|\varphi_t \|^2  \no\\
& \hspace{0.5cm}- \frac{C_0}{4} \| \nabla \varphi_t \|^2 -C \| \varphi_t\|^2 - C \|\nabla \J\|_{L^1}\|\varphi_t \|^2-C \| \varphi_t \|^4 \no \\
&\geq \frac{C_0}{2} \| \nabla \varphi_t \|^2 - C \|\varphi_t \|^2 - C \| \varphi_t \|^4-C. \label{31}
\end{align}
From the estimate
\begin{align} \label{30}
|(\nabla \mu_t ,\nabla \varphi_t)| \leq \|\nabla \mu_t \| \| \nabla \varphi_t\| \leq \frac{1}{C_0} \|\nabla \mu_t \|^2 + \frac{C_0}{4} \| \nabla \varphi_t \|^2
\end{align}
and \eqref{31}, we get
\begin{align}
\| \nabla \varphi_t \|^2 \leq \frac{4}{C_0^2} \|\nabla \mu_t \|^2+C \|\varphi_t \|^2 +C \|\varphi_t \|^4+C.\label{154}
\end{align}
Using H\"older inequality and Gagliardo-Nirenberg inequality and \eqref{154} we get,
\begin{align}
|(\mu \nabla \varphi_t, \u_t)| &=|(\nabla \mu, \varphi_t \u_t)| \no \\
&\leq \| \nabla \mu\|_{L^6} \| \varphi_t \|_{L^3} \| \u_t\| \no\\
&\leq C (1 + \|\varphi_t\|^{2/3})(\|\varphi_t\|^{2/3} \|\nabla\varphi_t\|^{1/3} + \|\varphi_t\| ) \| \u_t\|\no\\
&\leq C (\|\varphi_t\|^{2/3} \|\nabla\varphi_t\|^{1/3}+ \|\varphi_t\| ^{4/3}\|\nabla \varphi_t\|^{1/3}  + \|\varphi_t\| +\|\varphi_t\|^{5/3}  ) \| \u_t\|\no \\
&\leq \frac{1}{6} \| \u_t\|^2 +C (\|\varphi_t\|^{4/3} \|\nabla\varphi_t\|^{2/3}+ \|\varphi_t\| ^{8/3}\|\nabla \varphi_t\|^{2/3}  + \|\varphi_t\|^2 +\|\varphi_t\|^{10/3}  ) \no\\
&\leq \frac{1}{6} \|\u_t\|^2 +  \frac{C_0^2}{64} \| \nabla \varphi_t\|^2 + C \|\varphi_t\|^2  +\frac{C_0^2}{64} \| \nabla \varphi_t\|^2 + C \|\varphi_t\|^4 \no \\
& \hspace{2cm}+ \|\varphi_t\|^2 +\|\varphi_t\|^{10/3} \no \\
&\leq  \frac{1}{6} \|\u_t\|^2 +  \frac{C_0^2}{32} \| \nabla \varphi_t\|^2  + C \|\varphi_t\|^2 +C \|\varphi_t\|^4  + C \no \\
&\leq \frac{1}{6} \|\u_t\|^2 + \frac{1}{8} \|\nabla \mu_t \|^2 + C \|\varphi_t\|^2 + C\|\varphi_t\|^4 +C. \label{155}
\end{align}
Using integration by parts and Holder inequality we get,
\begin{align}
|(\u \cdot \nabla \varphi_t, \mu_t)| &=|(\u \varphi_t, \nabla \mu_t)| \no \\
&\leq \|\u\|_{L^\infty} \|\varphi_t\| \|\nabla \mu_t\|\no  \\
&\leq \frac{1}{4} \|\nabla \mu_t\|^2 + C \|\u\|^2_{\H^2} \|\varphi_t\|^2, \label{156}
\end{align}
\begin{align}
|\langle\h_t, \u_t\rangle| \leq \|\h_t\|_{\V'_{div}} \|\nabla \u_t\| \leq \frac{\nu}{2} \| \nabla \u_t\|^2 + C \|\h_t\|_{\V'_{\mathrm{div}}}^2, \label{157}
\end{align}
\begin{align}
|(\nabla \J*\varphi_t,\u_t\varphi)| &\leq \|\nabla \J*\varphi_t\|\|\u_t\|\|\varphi\|_{L^\infty} \no \\
&\leq \frac{1}{6} \|\u_t\|^2 + C \|\varphi\|^2_{\infty} \|\varphi_t\|^2,  \label{158}
\end{align}
\begin{align}
|(\nabla \J*\varphi_t, \u \varphi_t)| \leq \|\nabla \J * \varphi_t\| \|\u\|_{L^\infty} \|\varphi_t\| \leq C \|\u\|_{\H^2} \|\varphi_t\|^2, \label{159}
\end{align}
\begin{align}
|(\nabla \J*\varphi_t, \nabla \mu_t)| &\leq \|\nabla \mu_t\| \|\nabla \J* \varphi_t\|, \no \\
&\leq \frac{1}{4} \|\nabla \mu_t\|^2 + C  \|\varphi_t\|^2. \label{160}
\end{align}
Using Gagliardo-Nirenberg inequality we get
\begin{align}
\int_\Omega \F'''(\varphi) \varphi_t^3 &\leq C \|\varphi_t\|^3_{L^3} \leq C(\|\nabla \varphi_t\|^{1/3} \|\varphi_t\|^{2/3} +  \|\varphi_t\|)^3 \no \\
&\leq C (\|\nabla \varphi_t\| \|\varphi_t\|^2 + \|\varphi_t\|^3 )\no \\
& \leq \frac{C_0^2}{32} \|\nabla \varphi_t\|^2 + C \|\varphi_t\|^4 \no \\
&\leq \frac{C_0^2}{32}\left( \frac{4}{C_0^2} \|\nabla \mu_t\|^2 + C \|\varphi_t\|^4 +C \right) + C \|\varphi_t\|^4 \no \\
&\leq \frac{1}{8} \|\nabla \mu_t\|^2 + C \|\varphi_t\|^4 + C \label{161} .
\end{align}
Substituting above estimates \eqref{155}-\eqref{161} in \eqref{142}, we get 
\begin{align*}
&\frac{1}{2} \frac{\d}{\d t} \left(\int_\Omega (a + \F''(\varphi))\varphi_t^2 \right)+\nu \|\nabla \u_t\|^2 + \frac{1}{2} \|\u_t\|^2 + \frac{1}{4} \|\nabla \mu_t\|^2 \\
& \quad\leq C (\|\u\|^2_{\H^2} + 1  )\|\varphi_t\|^2 + C \|\varphi_t\|^4 +C \|\h_t\|_{\V'_{\mathrm{div}}}^2 + C.
\end{align*}
Hence we obtain
\begin{align} \label{35}
\frac{1}{2} \frac{\d}{\d t} \left(\int_\Omega (a + \F''(\varphi))\varphi_t^2 \right) \leq f(t) \|\varphi_t\|^2 + C \|\varphi_t\|^4 +C \|\h_t\|_{\V'_{\mathrm{div}}}^2 + C.
\end{align}
where $f(t) =C (\|\u\|^2_{\H^2} + 1  ) $. Let us multiply \eqref{35} by $(1+ \int_\Omega(a+\F''(\varphi))\varphi_t^2)^{-1}$. Then we get
\begin{align*}
\frac{1}{2} \frac{\d}{\d t} \log\left(1+\int_\Omega(a+\F''(\varphi))\varphi_t^2 \right) &\leq \frac{1}{C_0} f(t) + \frac{C \|\varphi_t\|^4}{1+\int_\Omega(a+\F''(\varphi)\varphi_t^2)} + C \|\h_t\|_{\V'_{\mathrm{div}}}^2 + C, \\
&\leq \frac{1}{C_0} f(t) + C \|\varphi_t\|^2 + C \|\h_t\|_{\V'_{\mathrm{div}}}^2 + C. 
\end{align*}
Now integrate from $0$ to $t$ to obtain
\begin{align*}
\frac{1}{2} \log \left( 1 + \int_\Omega (a + \F''(\varphi))\varphi_t^2 \right) &\leq \frac{1}{2} \log \left( 1 + \int_\Omega (a + \F''(\varphi_0))\varphi_t^2(0) \right) \\
&+ \frac{1}{C_0} \int_0^t f(t) + C \int_0^t \|\varphi_t\|^2 
+ C \int_0^t \|\h_t\|_{\V'_{div}}^2 +C .
\end{align*}
Since $f \in L^1(0,T)$, $\h_t \in L^2(0,T; \G_{\text{div}})$, and $\varphi(0) \in H^2(\Omega)$, from \eqref{36} we can say that
\begin{align} \label{38}
\varphi_t \in L^\infty(0,T; H), \quad \forall \,T>0.
\end{align}
Using \eqref{38} in \eqref{27} and \eqref{37} we have
\begin{align} \label{39}
\nabla \varphi , \nabla \mu \in L^\infty (0,T; L^p(\Omega)) \quad \forall \, T>0, \ 2 \leq p < \infty.
\end{align}
Substituting \eqref{38} and \eqref{39} in \eqref{143} we get that 
\begin{align} \label{141}
\Delta \mu \in L^\infty(0,T; L^2(\Omega)).
\end{align}
Hence, using \eqref{137},\eqref{39} and \eqref{141} in \eqref{140} we infer that
\begin{align*}
\varphi \in L^\infty (0,T; \mathrm{H}^2(\Omega)).
\end{align*}
\end{proof}

\begin{remark}
Substituting \eqref{136} and \eqref{11} in \eqref{9} we get 
\begin{align*}
\nu \| \Delta \u \|^2 + \| \nabla \u \|^2 \leq \frac{1}{2}\| \nabla a \|_{L^\infty} \| \varphi \|^2_{L^4} \| \Delta \u\| + \| \nabla \J \|_{L^1} \| \varphi\|_{L^4}^2  \| \Delta \u \|+\| \h \| \| \Delta \u \|.
\end{align*}
On using Sobolev inequality and \eqref{137}, this implies
\begin{align*}
\nu \|\Delta \u \| \leq C \| \varphi\|_V^2 +\| \h \| .
\end{align*}
Using \eqref{137}, we infer that 
\begin{align*}
\Delta \u \in L^\infty (0,T; \G_{\mathrm{div}}),
\end{align*}
and using Theorem \ref{regularity} we get
\begin{align*}
\u \in L^\infty(0,T; \H^2).
\end{align*}
\end{remark}

\begin{theorem} \label{diff1}
Suppose that the hypothesis (\textbf{H1})-(\textbf{H5}) is fulfilled. Let $\h_1, \h_2 \in \mathcal{U}$ and let $[ \varphi_1,\u_1]$ and $[\varphi_2,\u_2]$ be two unique strong solutions of the system \eqref{16}-\eqref{op2} corresponding to $\h_1$ and $\h_2$ respectively with the same initial data satisfying \eqref{43}-\eqref{42}. Then there exists a constant $C>0$ such that 
\begin{align}
 \|\varphi_1-\varphi_2\|^2_{L^\infty([0,t];H)}+\| \nabla (\varphi_1-\varphi_2)\|^2_{L^2(0,t;H)} +\|\u_1-\u_2\|^2_{L^2(0,t; \V_{\mathrm{div}})} \no \\
 \leq C \|\h_1-\h_2\|^2_{L^2(0,T;\V_{\mathrm{div}}')} \label{bound1}
\end{align}
for every $t \in [0,T]$.
\end{theorem}
\begin{proof}
 Set $\varphi= \varphi_1 - \varphi_2$, $\u= \u_1-\u_2$ and $\h = \h_1 - \h_2$, then we obtain
\begin{align}
\langle \varphi_t , \psi \rangle + (\nabla \mu, \nabla \psi) &= (\u \varphi_1, \nabla \psi) + (\u_2 \varphi, \nabla
\psi), \ \forall \psi \in V, \label{47} \\
\nu (\nabla \u , \nabla \v) + (\u , \v) &= (\mu \nabla \varphi_1, \v) + (\mu_2 \nabla \varphi, \v) + (\h, \v), \ \forall \v \in \V_{\mathrm{div}}, \label{48} \\
\varphi(0) &= \varphi_{01}-\varphi_{02},\no
\end{align}
where 
$ \mu = a \varphi - \mathrm{J}* \varphi + \F'(\varphi_1) - \F' (\varphi_2).$

Let us substitute $\v=\u$ and $\psi =\varphi$ in \eqref{47} and \eqref{48} respectively. We obtain
\begin{align}
\frac{1}{2} \frac{\d}{\d t}\|\varphi (t)\|^2 + (\u \nabla \varphi_1 , \varphi)+ (\u_2 \nabla \varphi, \varphi)&=  (\Delta \mu , \varphi),  \label{56}\\
\nu\|\nabla \u\|^2 + \|\u \|^2 &= (\mu \nabla \varphi_1,\u )+(\mu_2 \nabla \varphi,\u) +(\h,\u) \label{54}.
\end{align}
Now we estimate terms in \eqref{56}. We denote by $C=C(\J, \F, \nu, \Omega, C_0)$.
Using H\"older and Ladyzhenskaya inequalities
\begin{align}
|(\u \nabla \varphi_1 , \varphi)| &\leq \|\u\|_{L^4}\| \nabla \varphi_1\|_{L^4} \| \varphi\|\no \\
& \leq \|\nabla \u\| \| \nabla \varphi_1\|_{L^4} \| \varphi\| \no  \\
& \leq \frac{\nu}{5} \|\nabla \u\|^2 + C \| \nabla \varphi_1\|^2_{L^4} \| \varphi\|^2. \label{91}
\end{align}
Observe that, 
\begin{align*}
\int_\Omega \u_2 (\nabla \varphi) \varphi= \int_\Omega \u_2 \nabla \left(\frac{\varphi^2}{2} \right) = - \int_\Omega \mathrm{div}(\u) \frac{\varphi^2}{2} =0.
\end{align*}
Using (\textbf{H3}), we get
\begin{align}
-(\Delta \mu, \varphi)&= (\nabla \mu, \nabla \varphi) \no \\ 
&=(\nabla (a \varphi - \J*\varphi+ \F(\varphi_1)-\F(\varphi_2)), \nabla \varphi) \no \\
&=((a + \F''(\varphi_2) )\nabla \varphi, \nabla \varphi) +(\varphi \nabla a  - \nabla \J*\varphi ,\nabla \varphi) \no \\
& \quad+ ((\F''(\varphi_1) - \F''(\varphi_2)) \nabla \varphi_1,\nabla \varphi) \no \\
&\geq C_0\|\nabla \varphi\|^2  +(\varphi \nabla a  - \nabla \J*\varphi ,\nabla \varphi) + ((\F''(\varphi_1) - \F''(\varphi_2)) \nabla \varphi_1,\nabla \varphi) \label{135}
\end{align}
Right hand side terms of \eqref{135} can be estimated as follows
\begin{align}
|(\varphi \nabla a  - \nabla \J*\varphi ,\nabla \varphi) | &\leq \|\varphi\|\|\nabla a \| \|\nabla \varphi\| + \|\nabla \J\| \| \varphi\|\|\nabla \varphi\| \no \\
&\leq \frac{C_0}{8}\|\nabla \varphi\|^2 + \frac{2}{C_0} \|\nabla a \|^2 \|\varphi\|^2 + \frac{C_0}{8}\|\nabla \varphi\|^2 + \frac{2}{C_0}\|\nabla \J\|^2 \| \varphi\|^2 \no \\
&\leq \frac{C_0}{4}\|\nabla \varphi\|^2 + C\| \varphi\|^2, \label{138}
\end{align}
using Holder inequality and Gagliardo-Nirenberg inequality we get
\begin{align}
|((\F''(\varphi_1) - \F''(\varphi_2)) \nabla \varphi_1,\nabla \varphi)| &\leq C \|\varphi\|_{L^4} \|\nabla \varphi_1\|_{L^4} \|\nabla \varphi\| \no \\
& \leq C (\|\varphi\| + \|\varphi\|^{1/2} \|\nabla \varphi\|^{1/2} )\|\nabla \varphi_1\|_{L^4} \|\nabla \varphi\| \no \\
& \leq \frac{C_0}{8} \|\nabla \varphi\|^2 + C \|\nabla \varphi_1\|_{L^4}^2\|\varphi\|^2 \no \\
& \quad+ \frac{C_0}{8} \|\nabla \varphi\|^2 + C \|\nabla \varphi_1\|_{L^4}^4\|\varphi\|^2 \no \\
&\leq \frac{C_0}{4} \|\nabla \varphi\|^2 + C(\|\nabla \varphi_1\|_{L^4}^2 + \|\nabla \varphi_1\|_{L^4}^4) \|\varphi\|^2 , \label{139}
\end{align}
Substituting \eqref{138} and \eqref{139} in \eqref{135} we get
\begin{align}
-(\Delta \mu, \varphi) \geq \frac{C_0}{2} \|\nabla \varphi\|^2 -C(1+\|\nabla \varphi_1\|_{L^4}^2 + \|\nabla \varphi_1\|_{L^4}^4)\|\varphi\|^2. \label{92}
\end{align}
Using \eqref{91} and \eqref{92} in \eqref{56} we get 
\begin{align}
&\frac{1}{2} \frac{\d}{\d t}\|\varphi (t)\|^2+\frac{C_0}{2} \|\nabla \varphi\|^2 \no  \\
 & \qquad\leq \frac{\nu}{5} \|\nabla \u\|^2 +C(1+\|\nabla \varphi_1\|_{L^4}^2 + \|\nabla \varphi_1\|_{L^4}^4)\|\varphi\|^2. \label{99}
\end{align} 
Now we estimate the terms in \eqref{54}. Observe that
\begin{align}
\| {\mu} \| &= \| a \varphi - \mathrm{J}* \varphi + \F'(\varphi_1) - \F' (\varphi_2) \|\no \\
& =\| a \varphi - \mathrm{J}* \varphi + \F''(\varphi_1+\theta\varphi_2) \varphi\|\no \\
&\leq \|a\|_\infty \|\varphi\| + \| \J\|_{L^1}\|\varphi\| + C \|\varphi\| \no \\
&\leq C \|\varphi\|. \label{46}
\end{align}
Using \eqref{46} we get
\begin{align}
|(\mu \nabla \varphi_1,\u )| &\leq \|\mu\|\|\nabla \varphi_1\|_{L^4} \|\u\|_{L^4} \no \\
&\leq  \frac{\nu}{5}\|\nabla \u\|^2+ C \|\nabla \varphi_1\|^2\|\varphi\|^2, \label{95}
\end{align}
 and using integration by parts 
\begin{align}
|(\mu_2 \nabla \varphi,\u)| & =|(\varphi\nabla \mu_2  ,\u) | \no  \\
&\leq \|\varphi\|\|\nabla \mu_2\|_{L^4}\|\u\|_{L^4} \no  \\
&\leq \frac{\nu}{5} \|\nabla \u\|^2 + C \|\nabla \mu_2\|^2_{L^4}\|\varphi\|^2, \label{96}
\end{align}
also
\begin{align}
|(\h,\u)| &\leq \|\h\|_{\V_{\mathrm{div}}'}\|\nabla\u\| \no \\
&\leq \frac{\nu}{5}\|\nabla \u\|^2+ C \|\h\|^2 _{\V_{\mathrm{div}}'}. \label{97}
\end{align}
Combining \eqref{95}, \eqref{96} and \eqref{97} we get
\begin{align}
\frac{2\nu}{5}\|\nabla \u\|^2 + \|\u \|^2 = C(\|\nabla \varphi_1\|^2+\|\nabla \mu_2\|^2_{L^4}) \|\varphi\|^2 +C \|\h\|^2_{\V_{\mathrm{div}}'}. \label{98}
\end{align}
Adding \eqref{99} and \eqref{98}, we get
\begin{align}
\frac{1}{2} &\frac{\d}{\d t}  \|\varphi (t)\|^2 + \frac{C_0}{2} \|\nabla \varphi\|^2 +\frac{\nu}{5}\|\nabla \u\|^2 + \|\u \|^2 \no \\
&\leq C(1+\|\nabla \varphi_1\|_{L^4}^2 + \|\nabla \varphi_1\|_{L^4}^4+ \|\nabla \mu_2\|^2_{L^4})\|\varphi\|^2 
 + C \|\h\|^2_{\V_{\mathrm{div}}'}. \label{100}
\end{align}
By applying Gronwall's inequality
\begin{align*}
\|\varphi (t)\|^2 \leq C\exp \left( \int_0^t \alpha (s)ds \right)\|\h\|_{\L^2(0,T;\V_{\mathrm{div}}')},
\end{align*}
where $\alpha(t)=C(1+\|\nabla \varphi_1\|_{L^4}^2 + \|\nabla \varphi_1\|_{L^4}^4+ \|\nabla \mu_2\|^2_{L^4})\in L^1(0,T)$. Integrate \eqref{100} from $0$ to $t$ we arrive at
\begin{align*}
\|\varphi(t)\|^2+&\frac{C_0}{2} \int_0^t\|\nabla \varphi\|^2 +\frac{\nu}{5}\int_0^t\|\nabla \u\|^2 + \int_0^t\|\u \|^2 \\
& \qquad \leq C\left(\exp \left( \int_0^t \alpha (s) \d s \right)  \int_0^t \alpha (s) \d s + 1\right)\|\h\|_{\L^2(0,T;)\V_{\mathrm{div}}'}
\end{align*}
which gives 
\begin{align}
\|\u_1-\u_2\|^2_{L^2(0,t; \V_{\mathrm{div}})} +\|\varphi_1-\varphi_2\|^2_{C^0 (0,t;V)}+ \| \nabla (\varphi_1-\varphi_2)\|^2_{L^2(0,t;H)} \no  \\
 \leq C \|\h_1-\h_2\|^2_{\L^2(0,t;\V_{\mathrm{div}}')}.
\end{align}
for every $t \in [0,T]$. 

\end{proof}
Since for any $\h \in \mathcal{U}$ we have $\u \in L^\infty(0, T; \H^2)$, we can prove the following theorem for a higher-order estimate of $\varphi$ using the same techniques as in Lemma 2.6 in \cite{CHNS op} and estimates used in the proof of Theorem \ref{diff1}.
\begin{theorem}
Let us assume that (\textbf{H1})-(\textbf{H5}) are satisfied.  Let $\h_1, \h_2 \in \mathcal{U}$ and let $[ \varphi_1,\u_1]$ and $[\varphi_2,\u_2]$ be two unique strong solutions of the system \eqref{16}-\eqref{op2} corresponding to $\h_1$ and $\h_2$ respectively with the same initial data satisfying \eqref{43}-\eqref{42}. Then there exists a constant $C >0$ such that
\begin{align}
 \|\varphi_1-\varphi_2\|^2_{L^\infty (0,t;V)} + \|(\varphi_1)_t-(\varphi_2)_t\|^2_{L^2(0,t; H)} + \|\varphi_1-\varphi_2\|^2_{L^2(0,t;H^2)} \no \\
 +\|\u_1-\u_2\|^2_{L^2(0,t; \V_{\mathrm{div}})}\leq C \|\h_1-\h_2\|^2_{\L^2(0,T;\V_{\mathrm{div}}')}, \label{bound2}
\end{align}
for every $t \in (0,T]$.
\end{theorem}

	\section{Optimal control} \label{linear}
In this section we study the optimal control problem(OCP) related to \eqref{16}-\eqref{op2} defined as
 minimizing the tracking type cost functional $\mathcal{J}$
	\begin{align}
	\mathcal{J}(\varphi,\u, \U) & := \int_0^T \| \varphi(t) - \varphi_d(t) \|^2 \d t +\int_0^T \|\u(t) -\u_d(t) \|^2 \d t \no \\
& \hspace{1cm}	+ \int_\Omega |\varphi(x,T)-\varphi_\Omega|^2 \d x + \int_0^T \|\U(t)\|^2  \d t \label{cost}
	\end{align}
	in the bounded, closed and convex set of admissible controls 
	\begin{align} \label{controlset}
	\mathcal{U} _{ad} = \left\{ \U \in \mathcal{U} \quad | \quad \U_1 (x,t) \leq \U(x,t) \leq \U_2 (x,t), \mathrm{a.e.} (x,t) \in \Omega\times (0,T) \right\}
	\end{align}
	subject to the system
	\begin{align} \label{cont1}
		\varphi_t + \u \cdot \nabla \varphi & = \Delta \mu \quad \text{in} \quad \Omega \times (0,T) \\
	\mu &= a \varphi  - \J * \varphi +\F'(\varphi)   \quad \text{in} \quad \Omega \times (0,T)\\
	-\nu \Delta \u + \u + \nabla \pi &= \mu \nabla \varphi + \U \quad \text{in} \quad \Omega \times (0,T) \label{cont3} \\
	\text{div} \, (\u) & = 0  \quad \text{in} \quad \Omega \times (0,T)\\
	\u = \frac{\partial \mu}{\partial \n}&  = 0  \quad \text{on} \quad \partial\Omega \times (0,T)\\
	\varphi (0)  & = \varphi_0 (x) \quad \text{in} \quad \Omega, \label{cont2}
	\end{align}
	where the external force $\U$ plays the role of control and $[\varphi, \u]$ solves \eqref{cont1}-\eqref{cont2}. 
We assume that the desirable concentration $\varphi_d $ and  velocity $\u_d$  belong to $ L^2(\Omega \times (0,T))$ and $ L^2(0,T; \G_{\mathrm{div}})$ respectively. Moreover, $\varphi_\Omega \in L^2(\Omega)$ and $\U_1, \U_2 \in \mathcal{U} \cap \mathbb{L}^\infty(\Omega \times (0,T))$. The optimal control problem is defined as 
\begin{align}
\min_{ \U \in {\mathcal{U}_{ad}} } \; \; \{ \mathcal{J}(\varphi,\u, \U) | (\varphi,\u, \U) \mbox{ is unique strong solution of \eqref{cont1}-\eqref{cont2}}  \} \quad \quad {\mbox(OCP)} \nonumber
\end{align}

Let us define the control to state operator $S: U \rightarrow (\varphi, \u)$  where $(\varphi, \u)$ solves   \eqref{cont1}-\eqref{cont2} with control $\U$. Note that,
\begin{align*}
S : \mathcal{U} \rightarrow \mathcal{V}:= C([0,T]; H) \cap L^2(0,T;V) \times L^2(0,T; \V_{\mathrm{div}}). 
\end{align*}

From Theorem \ref{existence}, Theorem \ref{weakunique} and Theorem \ref{strongsol} we can say that $S$ is a well defined map from $\mathcal{U}$ to $\mathcal{V}$. In fact, it is locally Lipschitz continuous. 

In this section we prove three important results. First one is to prove the existence of optimal control for the problem (OCP) defined above. The next result deals with the existence of a solution for the linearised system, linearised around the optimal state. Further, we prove that the control to state operator $S$ identified above is differentiable, and the Fr\'echet derivative of $S$ is given in terms of the solution of the linearised system. 

\subsection{Existence of Optimal Control}
\begin{theorem}\label{existenceofOptimalControl}
Suppose that the hypothesis (\textbf{H1})-(\textbf{H5}) are satisfied and assume that the admissible set of controls  $\mathcal{U}_{ad} \subset \mathcal{U}$ be as given by \eqref{controlset}. Then optimal control problem (OCP) admits a solution.
\end{theorem}
\begin{proof}
Let us define $l= \inf_{\U \in \mathcal{U}_{ad}} \mathcal{J}(\varphi, \u, \U)$. Since $0 \leq l < \infty$, there exists a minimizing sequence $\{U_n\} \in \mathcal{U}_{ad}$ for  \eqref{cost} such that 
\begin{align*}
\lim_{n\rightarrow \infty} \mathcal{J}(\varphi_n, \u_n, \U_n) = l.
\end{align*}
 where $[\varphi_n, \u_n]=S(\U_n)$ is a corresponding state solution of the system \eqref{cont1}-\eqref{cont2}. Since $\mathcal{U}_{ad}$ is bounded in $\mathcal{U}$, using the embedding $\mathcal{U} \subset L^\infty(0,T; \G_{\mathrm{div}}) \subset L^2(0,T; \G_{\mathrm{div}})$ we get  
\begin{align*}
\{ \U_n \} \ \ \mathrm{ is \ uniformly \ bounded \ in \ } L^2(0,T;\G_{\mathrm{div}}).
\end{align*}
Since $\mathcal{U}$ is closed there exists $\U^* \in \mathcal{U}$ such that 
\begin{align*}
\U_n \xrightarrow{w} \U^* \ \mathrm{in} \  L^2(0,T;\G_{\mathrm{div}})
\end{align*} and using the estimates in Theorem 2.2 in \cite{CHB weak}  we get 
\begin{align*}
&\{ \varphi_n \} \ \mathrm{ is \ uniformly \ bounded \ in \ } L^\infty(0,T;H)\cap L^2(0,T;V), \\
&\{ \varphi_n' \} \ \mathrm{is \ uniformly \ bounded \ in} \ L^2(0,T;V'), \\
&\{ \mu_n \} \ \mathrm{is \ uniformly \ bounded \ in} \ L^2(0,T;V) ,\\
&\{ \u_n \} \ \mathrm{is \ uniformly \ bounded \ in} \ L^2(0,T;\V_{\mathrm{div}}).
\end{align*}
We can find sub-sequences (still denoted by same subscript) and $ \varphi^* \in L^\infty(0,T;H)\cap L^2(0,T;V), \u^* \in L^2(0,T;V)$ such that 
\begin{align*}
&\varphi_n \xrightarrow{w^*} \varphi^* \ \mathrm{in} \  L^\infty(0,T;H), \\
&\varphi_n \xrightarrow{w} \varphi^* \ \mathrm{in} \  L^2(0,T;V) ,\\
&(\varphi_n)_t \xrightarrow{w} \varphi^*_t \ \mathrm{in} \  L^2(0,T;V'),\\
&\u_n\xrightarrow{w} \u^* \ \mathrm{in} \  L^2(0,T;\V_{\mathrm{div}}). 
\end{align*}
By Aubin-Lions Compactness lemma we get 
\begin{align*}
\varphi_n \xrightarrow{s} \varphi^* \ \mathrm{in} \ C([0,T];H)
\end{align*}
which gives 
\begin{align*}
\mu_n = a \varphi_n -\J*\varphi_n + \F'(\varphi_n) \xrightarrow{s} a \varphi^* -\J*\varphi^* + \F'(\varphi^*) = \mu. 
\end{align*}
Using these convergences, we pass to the limit in the weak formulation of \eqref{cont1}-\eqref{cont2} like in \cite{CHB weak} written for each $n \in \mathbb{N}$ then we can see that $[ \varphi^*, \u^*] = S(\U^*)$. Since $\mathcal{J}$ is convex and continuous functional, it follows that $\mathcal{J}$ is weakly lower semi continuous. Hence we have 
\begin{align*}
\mathcal{J}(\varphi^*, \u^*, \U^*) \leq \liminf \mathcal{J}(\varphi_n, \u_n, \U_n)
\end{align*}
which implies 
\begin{align*}
l \leq \mathcal{J}(\varphi^*, \u^*, \U^*) \leq \liminf \mathcal{J}(\varphi_n, \u_n, \U_n) = \lim \mathcal{J}(\varphi_n, \u_n, \U_n)=l.
\end{align*}
Hence we conclude that $[\varphi^*, \u^*]$ is the optimal state with the optimal control $\U^*$.
\end{proof}	
	\subsection{Linearised system}
	Let $\U^*$ be an optimal control and $[ \varphi^*,\u^*]$ be corresponding strong solution of the system \eqref{cont1}-\eqref{cont2} in the sense of Theorem \ref{strongsol}. Let $\U \in \mathcal{U}$ be given. Consider the following system which is obtained by linearising the system \eqref{cont1}-\eqref{cont2} around the optimal state $[ \varphi^*,\u^*]$
\begin{align} 
	\psi_t + \w \cdot \nabla \varphi^* + \u^* \nabla \psi & = \Delta \tilde{\mu}, \quad \text{in} \quad \Omega \times (0,T) \label{2},\\
	\tilde{\mu} &= a \psi - \J * \psi +\F''(\varphi^*)\psi ,  \quad \text{in} \quad \Omega \times (0,T),\\
	-\nu \Delta \w +\w + \nabla \pi_{\w} &= \tilde{\mu} \nabla \varphi^* + \mu^* \nabla \psi + \U,\quad \text{in} \quad \Omega \times (0,T) ,\label{1}\\
	\text{div} \, (\w) & = 0,  \quad \text{in} \quad \Omega \times (0,T),\\
	\w = \frac{\partial \tilde{\mu}}{\partial \n}&  = 0 , \quad \text{on} \quad \partial\Omega \times (0,T),\\
	\psi (0)  & = \psi_0 (x), \quad \text{in} \quad \Omega, \label{lin1}
\end{align}	
where $\mu^*= a\varphi^*-\J*\varphi^* + \F'(\varphi^*)$.	
\begin{theorem} \label{thmlin}
Suppose that (\textbf{H1})-(\textbf{H5}) are satisfied. Then for every $\U \in \mathcal{U} $ there exists a unique weak solution for the problem \eqref{2}-\eqref{lin1} such that 
\begin{align*}
\psi \in C([0,T]; H) \cap L^2(0,T; V)\cap H^1 (0,T; V')), \ 
\w \in L^2(0,T; \V_{\mathrm{div}}).
\end{align*}
\end{theorem}
\begin{proof}
We prove the existence of solution for linearized system using Faedo-Galerkin approximation scheme using the method in \cite{CHB weak}. We consider the families of functions $(\eta_k)$ and $(\v_k)$, eigenfunctions of the operator $-\Delta+ I : D(\mathcal{B}) \rightarrow H $ and of the Stokes operator respectively. Now, define a finite dimensional subspaces $\Psi_n:=\langle \eta_1, \cdots ,\eta_n \rangle$ and $\mathcal{V}_n := \langle \v_1, \cdots \v_n \rangle$ spanned by first $n$ functions of respective spaces, and orthogonal projectors on this spaces,  $\tilde{P}_n := P_{\Psi_n}$ and $P_n:=P_{\mathcal{V}_n}$. Then we look for the functions 
\begin{align*}
\psi_n (t) := \sum_{i=1}^{n} a_i^{(n)} (t) \eta_i , \qquad \w_n(t) := \sum_{i=1}^n b_i^{(n)}(t) \v_i.
\end{align*}
as a solution of the following approximation 
\begin{align}
&\langle (\psi_n)_t(t), \eta_i \rangle + (\w_n(t) \cdot \nabla \varphi^*(t), \eta_i ) + (\u^*(t) \cdot \nabla \psi_n(t), \eta_i)= - (\nabla \tilde{\mu}_n(t), \nabla \eta_i), \label{63}\\
&\nu (\nabla \w_n(t), \nabla \v_i) + (\w_n(t), \v_i)= (\tilde{\mu}_n(t) \nabla \varphi^*(t), \v_i) + (\mu^*(t) \nabla \psi_n(t), \v_i) + (\U(t), \v_i), \label{62}\\
& \hspace{5cm} \psi_n(0) = 0. \label{153} 
\end{align}
for $i=1, \cdots n$. This is nothing but a Cauchy problem for a system of 2$n$ ordinary differential equations in the $n$ unknowns $a_i^{(n)}$ and $b_i^{(n)}$ .  Using the Cauchy-Lipschitz theorem there exists a unique solution $(\psi_n, \w_n)$ to the approximated system. Now, multiply \eqref{63} and \eqref{62} by $a_i^{(n)}$ and $b_i^{(n)}$ respectively, and sum over $i=1, \cdots ,n$. We get 
\begin{align}
\nu \|\nabla \w_n\|^2 + \|\w_n\|^2+ \frac{\d}{\d t}\|\psi_n\| + (\w_n \cdot \nabla \varphi^*, \psi_n ) + (\u^* \cdot \nabla \psi_n, \psi_n)+(\nabla \tilde{\mu}_n, \nabla \psi_n) \no \\
= (\tilde{\mu}_n \nabla \varphi^*, \w_n) + (\mu^* \nabla \psi_n, \w_n) + (\U(t), \w_n) \label{69}
\end{align}
Now we have the following estimates,
\begin{align}
|(\w_n \cdot \nabla \varphi^*, \psi_n )| &\leq \|\w_n\|_{L^4} \|\nabla \varphi^* \|_{L^4} \|\psi\| \no \\
&\leq \|\nabla \w_n\| \|\nabla \varphi^*\|_{L^4} \|\psi\| \no \\
&\leq \frac{\nu}{4} \|\nabla \w_n\|^2 + \frac{1}{\nu} \|\nabla \varphi^* \|^2_{L^4} \|\psi\|^2, \label{64}
\end{align}
\begin{align}
(\u^* \cdot \nabla \psi_n, \psi_n) &\leq \|\u^*\|_{L^\infty} \| \nabla \psi_n\|\|\psi_n\| \no \\
&\leq \frac{C_0}{8} \|\nabla \psi_n\|^2 + C \|\u^*\|^2_{H^2} \|\psi_n\|^2 \label{70}
\end{align}
Using (\textbf{H3}), we get
\begin{align}
(\nabla \tilde{\mu}_n, \nabla \psi_n) &= (\nabla (a \psi_n-\J*\psi_n + \F''(\varphi^*)\psi_n), \nabla \psi_n) \no \\
 &= (a\nabla \psi_n + \psi_n \nabla a -\nabla \J*\psi_n + \F''(\varphi^*) \nabla \psi_n+ \F'''(\varphi^*) \nabla \varphi^* \psi_n, \nabla \psi_n)\no \\
&\geq C_0 \|\nabla \psi_n\|^2 + (\psi_n \nabla a , \nabla \psi_n)-(\nabla \J*\psi_n, \nabla \psi_n ) + (\F'''(\varphi^*)\nabla \varphi^* \psi_n, \nabla \psi_n), \label{149}
\end{align}
To estimate the right hand side terms of \eqref{149}
\begin{align}
(\psi_n \nabla a , \nabla \psi_n) &\leq \|\psi_n\|\|\nabla a\|_{\infty}\|\nabla \psi_n\| \no \\
&\leq \frac{C_0}{4} \|\nabla \psi_n\|^2 + \frac{1}{C_0}\|\psi_n\|^2\|\nabla a\|^2_{\infty}, \label{66}
\end{align}
\begin{align}
(\nabla \J*\psi_n, \nabla \psi_n ) &\leq \|\nabla \J\|_{L^1} \|\psi_n\| \|\nabla \psi_n\| \no \\
&\leq \frac{C_0}{4} \|\nabla \psi\|^2 + \frac{1}{C_0} \|\nabla \J\|^2_{L^1} \|\psi_n\|^2, \label{67}
\end{align}
using Gagliardo-Nirenberg inequality
\begin{align}
(\F'''(\varphi^*)\nabla \varphi^* \psi_n, \nabla \psi_n) &\leq C \|\nabla \varphi^*\|_{L^4} \|\psi_n\|
_{L^4} \| \nabla \psi_n\| \no \\
& \leq C \|\nabla \varphi^*\|_{L^4} (\|\psi_n\|+\|\psi_n\|^{1/2} \|\nabla \psi_n\|^{1/2} ) \| \nabla \psi_n\| \no \\
&\leq \frac{C_0}{4} \| \nabla \psi_n\|^2 + C \|\nabla \varphi^*\|^2_{L^4} \|\psi_n\|^2 + C \|\nabla \varphi^*\|^4_{L^4} \|\psi_n\|^2. \label{150}
\end{align}
Substituting \eqref{66}-\eqref{150} in \eqref{149} we get 
\begin{align}
(\nabla \tilde{\mu}_n, \nabla \psi_n) 
&\geq \frac{C_0}{4}\|\nabla \psi_n\|^2 - C \|\psi_n\|^2. \label{68}
\end{align}
Now to estimate right hand side terms of \eqref{69}, using \eqref{46}, we get
\begin{align}
|(\tilde{\mu}_n \nabla \varphi^*, \w_n)| &= \| \tilde{\mu}_n\|\| \nabla \varphi^*\|_{L^4} \|\w_n\|_{L^4} \no \\
& \leq C \|\varphi\| \|\nabla \varphi^*\|_{L^4} \|\nabla \w_n\| \no \\
& \leq \frac{\nu}{4} \|\nabla \w_n\|^2 + C \|\varphi\|^2 \|\nabla \varphi^*\|^2_{L^4}. \label{151}
\end{align}
\begin{align}
|(\mu^* \nabla \psi_n, \w_n)| &= |(\psi_n \nabla \mu^*  , \w_n) | \no \\
&\leq \| \nabla \mu^*\|_{L^4} \|\psi_n\| \|\w_n\|_{L^4} \no \\
&\leq \| \nabla \mu^*\|_{L^4} \|\psi_n\| \|\nabla \w_n\| \no \\
&\leq \frac{\nu}{4} \|\nabla \w_n\|^2 + \frac{1}{\nu} \| \nabla \mu^*\|^2_{L^4} \|\psi_n\|^2, \label{65}
\end{align}
\begin{align}
|(\U(t), \w_n)| \leq \|\U\| \| \w_n\| \leq \frac{1}{2} \|\w_n\|^2 + \frac{1}{2} \|\U\|^2 \label{152}
\end{align}
Substituting \eqref{64}, \eqref{70} and \eqref{68}-\eqref{152}  in \eqref{69} we arrive at
\begin{align} \label{71}
\frac{\d}{\d t}\|\psi_n\|^2 + \frac{\nu}{4} \|\nabla \w_n\|^2 + \frac{1}{2}\|\w_n\|^2+& \frac{C_0}{8} \|\nabla \psi_n\|^2 \no \\
\leq &C(1 + \|\u^*\|^2_{\H^2}) \|\psi_n\|^2 + \frac{1}{2} \|\U\|^2
\end{align}
By employing the Gronwall's lemma we conclude that
\begin{align*}
\|\psi_n(t)\|^2 \leq \exp\left( C \int_0^t (1 + \|\u^*(s)\|^2_{\H^2})d s\right) \left(\int_0^t \|\U(s)\|^2 ds \right)
\end{align*}
for all $t \in [0,T]$. Since $\u^* \in L^2(0,T; \H^2)$ we have 
\begin{align} \label{lin4}
\|\psi_n\|_{L^\infty(0,T; H)} \leq C \|\U\|_{\mathcal{U}} 
\end{align}
and integrating \eqref{71} from $0$ to $t$, we get 
\begin{align}\label{lin5}
\|\psi_n\|_{L^2(0,T;V)} \leq C \|\U\|_{\mathcal{U}} , \\
\| \w_n\|_{L^2(0,T; \V_{\mathrm{div}})} \leq C \|\U\|_{\mathcal{U}}, \label{lin6}
\end{align}
 moreover, from \eqref{63} we get 
that 
\begin{align*}
\|(\psi_n)_t\|_{V'} \leq C( \|\nabla \w_n \| \|\nabla \varphi^*\| + \|\nabla \u^*\| \|\nabla \psi_n\| + C \|\psi_n\|_V + \|\nabla \varphi^*\|_{L^4}\|\psi_n\|_V),
\end{align*}
that is 
\begin{align}
\|(\psi_n)_t\|_{L^2(0,T;V')}  \leq C \|\U\|_{\mathcal{U}}, \label{lin7}
\end{align}
for every $n \in \mathbb{N}$. From above uniform bounds we can obtain sub-sequences of $\{\psi_n\},\{(\psi_n)_t\}$ and $\{\w_n\}$, again denoted by $\{\psi_n\},\{(\psi_n)_t\}$ and $\{\w_n\}$ and functions $\psi \in L^\infty(0,T; H) \cap L^2(0,T;V) $, $\psi_t \in L^2(0,T;V') $ and $\w \in L^2(0,T; \V_{\mathrm{div}})$ such that
\begin{align*}
&\psi_n \xrightharpoonup{w^*} \psi \quad \text{in} \quad L^\infty (0,T; H),  \\
&\psi_n \xrightharpoonup{w} \psi \quad \text{in} \quad L^2 (0,T; V), \\
&(\psi_n)_t \xrightharpoonup{w} \psi_t \quad \text{in} \quad L^2 (0,T; V'), \\
&\w_n \xrightharpoonup{w} \w \quad \text{in} \quad L^2(0,T; \V_{\mathrm{div}}). 
\end{align*}
By passing to the limit in \eqref{63}-\eqref{153} we can say that there exists a weak solution $(\w, \psi) \in L^2(0,T; \V_{\mathrm{div}}) \times L^\infty (0,T; H) \cap L^2 (0,T; V)\cap H^1 (0,T; V')$. By Aubin's compactness lemma we have that $\psi_n \rightarrow \psi$ in $L^2(0,T;H)$ and $\psi \in C([0,T];H)$. This gives,
\begin{align*}
(\psi,\w) \in (C([0,T];H) \cap L^2 (0,T; V)\cap H^1 (0,T; V')) \times L^2(0,T; \V_{\mathrm{div}}).
\end{align*}
To prove that the solution $[\psi, \w]$ of  \eqref{2}-\eqref{lin1} is unique, let $[\psi_1,\w_1]$ and $[\psi_2, \w_2]$ be any two solutions of \eqref{2}-\eqref{lin1}. Let $\psi=\psi_1-\psi_2$ and $\w=\w_1-\w_2$. Then $[\psi, \w]$ satisfies 
\begin{align}
\psi_t + \w \cdot \nabla \varphi^* + \u^* \nabla \psi & = \Delta \tilde{\mu}, \label{lin2}\\
\tilde{\mu} &= a \psi - \J * \psi +\F''(\varphi^*)\psi,\\
	-\nu \Delta \w +\w + \nabla \tilde{\pi} &= \tilde{\mu} \nabla \varphi^* + \mu^* \nabla \psi  \label{lin3},\\
	\text{div} \, (\w) & = 0,  \\
	\w|_{\partial \Omega} = \frac{\partial \tilde{\mu}}{\partial \n} |_{\partial \Omega}&  = 0, \\
	\psi (0)  & = \psi_0 (x). 
\end{align}
We take inner product of \eqref{lin2} and \eqref{lin3} with $\psi$ and $\w$ respectively. Making use of the estimates derived for proving \eqref{69} we can prove that the weak solution of the system \eqref{2}-\eqref{lin1} is unique. From the estimates \eqref{lin4}-\eqref{lin7} we can also conclude that the mapping $\U \mapsto [\psi,\w]$ is a continuous linear mapping from $\mathcal{U}$ to $ L^\infty(0,T;H) \cap L^2(0,T;V) \cap H^1(0,T; V')\times L^2(0,T; \V_{\mathrm{div}})$.
\end{proof}

\subsection{Optimal condition} \label{control state}
In this section, we prove the differentiability of the control-to-state operator. For,  we also need  following assumption on $\F$  namely, \\
(\textbf{H7}) $\F \in C^4(\mathbb{R})$. 
\begin{theorem} \label{frech}
Suppose that the hypothesis (\textbf{H1})-(\textbf{H5}) and (\textbf{H7}) are satisfied. Then the control-to-state operator $ S: \mathcal{U} \rightarrow \mathcal{V}$ is Fr\'echet differentiable. 
Moreover, for any $\tilde{\U } \in \mathcal{U}$, the Fr\'echet derivative $S'$ at $\tilde{\U}$ in the direction of $\U$ is given by 
\begin{align*}
S'(\tilde{\U})(\U) = (\tilde{\psi},\tilde{\w}),
\end{align*}
for every $\U \in \mathcal{U}$, where $[\tilde{\w}, \tilde{\psi}]$ is the unique weak solution of the linearised system with control $\U$, which is linearised around strong solution of the controlled system \eqref{cont1}-\eqref{cont2} with control $\tilde{\U}$.

\end{theorem}
\begin{proof}
Note that the statement of the theorem states that for the optimal control $ \U^*$ whose existence is proved in the Theorem \ref{existenceofOptimalControl},
\begin{align*}
S'(\U^*)(\U) = (\psi, \w),
\end{align*}
for every $\U \in \mathcal{U}$, where $[\psi, \w]$ is the unique weak solution of the linearised system \eqref{2}-\eqref{lin1} with control $\U$. The Fr\'echet derivative at optimal control is going to be useful for us to characterize first-order optimality condition. Hence we will prove the theorem for optimal control $\U^*$. The general statement as stated above for any other $\tilde{\U} $, can be proved analogously.

 Let $[\varphi^*,\u^*]=S(\U^*)$ be the solution to the system \eqref{cont1}-\eqref{cont2} with control $\U^*$.
Let $[ \bar{\varphi},\bar{\u}]$ be solution of the system \eqref{cont1}-\eqref{cont2} with $\U^* + \U$. Let $\xi=\bar{\varphi}-\varphi^*,\ \z = \bar{\u} -\u^*$. Then $[\xi, \z]$ satisfies,
\begin{align*}
\xi_t + \z \cdot \nabla \xi + \z \cdot \nabla \varphi^*+ \u^* \nabla \xi & = \Delta \mu_\xi, \\
\mu_\xi &= a \xi -\J * \xi + \F'(\bar{\varphi})- \F' (\varphi^*) ,\\
-\nu \Delta \z + \z + \nabla \pi_\z& = \mu_\xi \nabla \varphi^* + \mu_\xi \nabla \xi + \mu^* \nabla \xi +  \U, \\
\text{div} \, \z & =0, \\
\z|_{\partial \Omega} = 0 , \ \frac{\partial \mu_\xi}{\partial \n}|_{\partial \Omega} & =0 \\
\xi (0) & = 0.
\end{align*}
where $ \pi_z=\pi_{\bar{\u}}-\pi_{\u^*}$. Now let us define $\rho=\xi- \psi , \ \y=\z - \w$ where $[\psi, \w]$ is the solution of the linearised system \eqref{2}-\eqref{lin1} corresponding to $\mathcal{\U}$. Then $(\y, \rho)$ satisfies 
\begin{align} 
\rho_t  + \y \cdot \nabla\varphi^* +\u^* \cdot \nabla \rho + &\z \cdot \nabla \xi  =\Delta \mu_\rho, \label{13} \\
\mu_\rho &=a \rho - \J*\rho+ \F'(\bar{\varphi}) - \F'(\varphi^*)- \F''(\varphi^*) \psi, \\
-\nu \Delta \y + \y + \nabla \pi_\y& = \mu_\rho \nabla \varphi^* + (a \xi- \J*\xi + \F'(\bar{\varphi}) -\F'(\varphi^*))\nabla \xi ,\no \\
& \quad + (a\varphi^* - \J*\varphi^* + \F'(\varphi^*)) \nabla \rho ,\label{12}\\
\text{div} \, \y & =0 ,\\
\y &= 0 , \ \frac{\partial \mu_\rho}{\partial \n}  =0, \\
\rho (0) & = 0.
\end{align}
where $\pi_y= \pi_{\bar{\u}}-\pi_{\u^*}-\pi_w$ with $\pi_{\bar{\u}}$ and $\pi_{\u^*}$ are the pressure terms appearing in \eqref{cont3} for $\U^*+ \U$ and $\U^*$ respectively and $\pi_w$ is the pressure term appearing in \eqref{1}. Now our aim is to prove that 
\begin{align}
\frac{\|[\rho, \y] \|_{\mathcal{V}}}{\|\U\|_{\mathcal{U}}} \rightarrow 0 \quad \mathrm{as} \quad \|\U\|_{\mathcal{U}} \rightarrow 0
\end{align}
 For, take inner product of \eqref{12} with $\y$ and of \eqref{13} with $\rho$ to get
\begin{align} \label{72}
\nu \| \nabla \y \|^2 + \| \y \|^2 =& (\mu_\rho \nabla \varphi^*, \y)+ ((a \xi- \J*\xi + \F'(\bar{\varphi}) -\F'(\varphi^*)) \nabla \xi, \y) \no \\ 
& + ((a \varphi^* - \J*\varphi^* + \F'(\varphi^*)) \nabla \rho, \y) 
\end{align}
\begin{align} \label{73}
\frac{1}{2} \frac{\d}{\d t} \| \rho (t) \|^2 + (\y \cdot \nabla \varphi^*, \rho )+ (\u^* \cdot \nabla \rho, \rho  ) + (\z \cdot \nabla \xi, \rho) \no \\
 =(\Delta (a \rho-\J*\rho+\F'(\bar{\varphi}) - \F'(\varphi^*)-\F''(\varphi^*) \psi), \rho ) 
\end{align}
We estimate right hand side terms of \eqref{72} one by one,
\begin{align*}
(\mu_\rho \nabla \varphi^*, \y) & = ((a \rho - \J * \rho +  \F'(\bar{\varphi})-\F(\varphi^*)-\F''(\varphi^*) \psi) \nabla \varphi^* , \y ) \\
& =  (a \rho  \nabla \varphi^*, \y)- ((\J * \rho)\nabla \varphi^*, \y)- ((  \F'(\bar{\varphi})-\F(\varphi^*)-\F''(\varphi^*) \psi) \nabla \varphi^*, \y) \\
& = I_1 + I_2 + I_3.
\end{align*}
\begin{align}
|I_1|  &\leq \| a\|_{\L^\infty} \| \rho \| \| \nabla \varphi^*\|_{\L^4} \| \y\|_{\L^4} \no \\
&\leq\| a\|_{\L^\infty} \| \rho \| \| \nabla \varphi^*\|_{L^4} \| \nabla \y\| \no \\
&\leq \frac{\nu}{10} \| \nabla \y\|^2 + C \| \nabla \varphi^*\|^2_{L^4} \| \rho \|^2, \label{76}
\end{align} 
\begin{align}
|I_2| &= \|\J * \rho\| \|\nabla \varphi^*\|_{\L^4} \| \y\|_{\L^4} \no\\
& \leq \|\J \|_{\L^1} \| \rho \| \|\nabla \varphi^*\|_{L^4} \| \nabla \y \|\no \\
& \leq \frac{\nu}{10} \| \nabla \y \|^2 + C  \|\nabla \varphi^*\|^2_{L^4} \| \rho \|^2, \label{77}
\end{align}
Notice that, since $\rho= \bar{\varphi}-\varphi^*-\psi$, using Taylor series we can write
\begin{align}\label{74}
  \F'(\bar{\varphi})-\F'(\varphi^*)-\F''(\varphi^*) \psi &= \F'(\bar{\varphi})-\F(\varphi^*)-\F''(\varphi^*) (\bar{\varphi}- \varphi^*) + \F''(\varphi^*) \rho \no \\
  & = \frac{1}{2}\F'''(\theta \bar{\varphi} + (1-\theta) \varphi^*)(\bar{\varphi}-\varphi^*)^2 + \F''(\varphi^*) \rho,
\end{align}
where $\theta \in (0,1)$. Then
\begin{align}
|I_3| &= |(\frac{1}{2}\F'''(\theta \bar{\varphi} + (1-\theta) \varphi^*)\xi^2 \nabla \varphi^*, \y) + (\F''(\varphi^*) \rho \nabla \varphi^*, \y)| \no \\
& \leq C_F \|\xi^2\| \| \nabla \varphi^*\|_{L^4} \|\y\|_{L^4} + C_F \|\rho\| \|\nabla \varphi^*\|_{L^4} \|\y\|_{L^4} \no \\
& \leq C_F \|\xi\|^2_{L^4} \| \nabla \varphi^*\|_{L^4} \|\nabla\y\| + C_F \|\rho\| \|\nabla \varphi^*\|_{L^4} \|\nabla \y\| \no \\
& \leq \frac{\nu}{10} \|\nabla \y\|^2 + C \|\xi\|^4_V \| \nabla \varphi^*\|^2_{L^4} + \frac{\nu}{10} \|\nabla \y\|^2 + C \|\rho\|^2 \|\nabla \varphi^*\|^2_{L^4}. \label{78}
\end{align}
Combining \eqref{76}, \eqref{77} and \eqref{78} we get
\begin{align}
|(\mu_\rho \nabla \varphi^*, \y)| \leq \frac{2 \nu}{5} \|\nabla \y\|^2 + C \| \nabla \varphi^*\|^2_{L^4}\| \rho \|^2 + C \|\xi\|^4_V \| \nabla \varphi^*\|^2_{L^4}. \label{79}
\end{align}
Using integration by parts, we estimate the second term on the right-hand side of \eqref{72} as follows
\begin{align*}
((a \varphi^* - \J*\varphi^* + \F'(\varphi^*)) \nabla \rho, \y) = -((a \nabla \varphi^* + \varphi^*\nabla a -\nabla \J *\varphi^*  + \F''(\varphi^*) \nabla \varphi^*) \rho,\y),
\end{align*}
\begin{align*}
|(a \nabla \varphi^* \rho , \y)| &\leq \|a\|_{\infty} \|\nabla \varphi^*\|_{L^4} \|\rho\| \|\y\|_{L^4} \\
& \leq \frac{\nu}{20} \|\nabla \y\|^2 + C \|\nabla \varphi^*\|^2_{L^4} \|\rho\|^2,
\end{align*}
\begin{align*}
|(\varphi^* \rho \nabla a ,\y)| &\leq \|\varphi^*\|_{L^4}\|\rho\| \|\nabla a \|_{\infty} \|\y\|_{L^4} \\
& \leq \frac{\nu}{20} \|\nabla \y\|^2 + C \|\varphi^*\|^2_{L^4} \|\rho\|^2,
\end{align*}
\begin{align*}
|((\nabla \J*\varphi^* )\rho, \y)| &\leq \|\nabla \J\|_{L^1} \|\varphi^*\|_{L^4} \|\rho\| \|\y\|_{L^4} \\
&\leq \frac{\nu}{20} \|\nabla \y\|^2 + C \|\varphi^*\|_{V} \|\rho\|^2,
\end{align*}
\begin{align*}
|(\F''(\varphi^*) \nabla \varphi^* \rho,\y)| &\leq C_F \|\nabla \varphi^* \|_{L^4} \|\rho\| \|\y\|_{L^4} \\
&\leq \frac{\nu}{20} \|\nabla \y\|^2 + C \|\nabla \varphi^*\|^2_{L^4} \|\rho\|^2,
\end{align*}
which gives
\begin{align}
|((a \varphi^* - \J*\varphi^* + \F'(\varphi^*)) \nabla \rho, \y)| \leq \frac{\nu}{5} \|\nabla \y\|^2 + C (\|\varphi^*\|_{V}^2 + \|\nabla \varphi^*\|_{L^4}^2)\|\rho\|^2. \label{80}
\end{align}
The third term can be estimated as 
\begin{align}
|((a \xi- \J* & \xi + \F'(\bar{\varphi}) -\F'(\varphi^*)) \nabla \xi, \y)| \no \\
& =|(\nabla a \frac{\xi^2}{2}, \y)| +|((\nabla \J * \xi) \xi,\y)|+ |(\F''(\theta \bar{\varphi} + (1-\theta)\varphi^*)\xi  \nabla \xi, \y) | \no \\
& \leq \|\nabla a \|_{\infty} \|\xi^2\| \|\y\| + \|\nabla \J\|_{L^1} \|\xi\|_{L^4} \| \xi\| \|\y\|_{L^4} +C_F \|\xi\|_{L^4} \|\nabla \xi\| \|\y\|_{L^4} \no \\
& \leq \frac{\nu}{20} \|\nabla \y\|^2 + C \|\xi\|^4_{L^4}+ \frac{\nu}{20} \|\nabla \y\|^2 + C \|\xi\|^4_{V} + \frac{\nu}{10} \|\nabla \y\|^2 + C \|\xi\|^4_V \no \\
& \leq \frac{\nu}{5} \|\nabla \y\|^2 + C \|\xi\|_V^4 .\label{81}
\end{align}
Substituting \eqref{79}, \eqref{80} and \eqref{81} in \eqref{72} we get
\begin{align*}
\nu \| \nabla \y \|^2 + \| \y \|^2 & \leq \frac{2 \nu}{5} \|\nabla \y\|^2 + C \| \nabla \varphi^*\|^2_{L^4}\| \rho \|^2 + C \|\xi\|^4_{V} \| \nabla \varphi^*\|^2_{L^4} \\
&+ \frac{\nu}{5} \|\nabla \y\|^2 + C (\|\varphi^*\|_{V}^2 + \|\nabla \varphi^*\|_{L^4}^2)\|\rho\|^2+ \frac{\nu}{5} \|\nabla \y\|^2 + C \|\xi\|_V^4,
\end{align*} 
which implies 
\begin{align}
\frac{\nu}{5} \|\nabla \y\|^2 + \|\y\|^2 \leq C  (\| \nabla \varphi^*\|^2_{L^4}+\|\varphi^*\|_{V}^2 )\| \rho \|^2  + C \|\xi\|^4_{L^4} \| \nabla \varphi^*\|^2_{L^4} 
+ C \|\xi\|_V^4. \label{85}
\end{align} 
We estimate the terms in \eqref{73} as follows
\begin{align}
|(\y \cdot \nabla \varphi^*, \rho )| & \leq \| \y \|_{\L^4} \| \nabla \varphi^* \|_{\L^4} \| \rho \| \no \\
& \leq \frac{\nu}{10} \| \nabla \y \|^2 + C \| \nabla \varphi^* \|^2_{L^4} \| \rho \|^2, \label{82}
\end{align}
Observe that, 
\begin{align*}
\int_\Omega (\u^* \cdot \nabla \rho) \rho \ \d x = \int_\Omega \u^* \cdot \nabla\left(\frac{\rho^2}{2} \right) \d x = - \int_\Omega (\textrm{div}(\u^*)) \frac{\rho^2}{2} =0,
\end{align*}
\begin{align}
|(\z \cdot \nabla \xi, \rho)| &\leq \| \z \|_{\L^4} \| \nabla \xi\| \| \rho \|_{\L^4} \no  \\
& \leq C_\Omega \| \nabla \z \| \| \nabla \xi \|(\|\nabla \rho\| + \|\rho\|) \no \\
& \leq \frac{C_0}{10} \| \nabla \rho \|^2 + C \| \nabla \z \|^2 \| \nabla \xi \|^2 + \frac{1}{2}\|\rho\|^2 + C\| \nabla \z \|^2 \| \nabla \xi \|^2 \no \\
& \leq \frac{C_0}{10} \| \nabla \rho \|^2 +\frac{1}{2}\|\rho\|^2+ C \| \nabla \z \|^2 \| \nabla \xi \|^2, \label{83}
\end{align}
From \eqref{74}, using \eqref{H3} we can write
\begin{align} \label{75}
&(\Delta (a \rho-\J*\rho+\F'(\bar{\varphi}) - \F'(\varphi^*)-\F''(\varphi^*) \psi), \rho )\hspace{4cm} \no \\
  &=-(\nabla (a \rho-\J*\rho+\frac{1}{2}\F'''(\theta \bar{\varphi} + (1-\theta) \varphi^*)\xi^2 + \F''(\varphi^*) \rho), \nabla \rho )\no \\
  &=-(a \nabla \rho + \rho \nabla a,\nabla \rho )+( \nabla \J * \rho , \nabla \rho) -(\F'''(\theta \bar{\varphi} + (1-\theta) \varphi^*)\xi \nabla \xi, \nabla \rho)  \no \\
  & \quad - (\frac{1}{2}\F^{(4)}(\theta \bar{\varphi} + (1-\theta) \varphi^*)((\theta \nabla \bar{\varphi} + (1-\theta)\nabla \varphi^*))\xi^2 , \nabla \rho)- (\F'''(\varphi^*) \nabla \varphi^* \rho, \nabla \rho)\no \\
  & \quad - (\F''(\varphi^*) \nabla\rho, \nabla \rho ) \no \\
  & \leq -C_0 \|\nabla \rho\|^2-(\rho \nabla a,\nabla \rho )+( \nabla \J * \rho , \nabla \rho)  -(\F'''(\theta \bar{\varphi} + (1-\theta) \varphi^*)\xi \nabla \xi, \nabla \rho) \no \\
  & \quad - (\frac{1}{2}\F^{(4)}(\theta \bar{\varphi} + (1-\theta) \varphi^*)((\theta \nabla \bar{\varphi} + (1-\theta)\nabla \varphi^*))\xi^2 , \nabla \rho)- (\F'''(\varphi^*) \nabla \varphi^* \rho, \nabla \rho).
\end{align}
Now we estimate right hand terms of \eqref{75} using H\"older and Sobolev inequalities 
\begin{align} \label{144}
|(\rho \nabla a,\nabla \rho )| &\leq \|\nabla a\|_{\infty} \|\rho\| \|\nabla \rho \|\no \\
& \leq \frac{C_0}{10} \|\nabla \rho\|^2 + C \|\rho\|^2,
\end{align}
\begin{align} \label{145}
|( \nabla \J * \rho , \nabla \rho)| &\leq \|\nabla \J\|_{L^1} \| \rho\| \|\nabla \rho\| \no\\
& \leq \frac{C_0}{10} \|\nabla \rho\|^2 + C \|\rho\|^2,
\end{align}
\begin{align} \label{146}
|(\frac{1}{2}\F^{(4)}(\theta \bar{\varphi} + (1-\theta) \varphi^*)((\theta \nabla \bar{\varphi} & + (1-\theta)\nabla \varphi^*))\xi^2 , \nabla \rho)|\no \\
&\leq C_F (\|\nabla\bar{\varphi}\|_{L^4}+ \|\nabla \varphi^*\|_{L^4})\|\xi^2\|_{L^4} \|\nabla \rho\| \no\\
&\leq \frac{C_0}{10} \|\nabla \rho\|^2 + C(\|\nabla\bar{\varphi}\|_{L^4}+ \|\nabla \varphi^*\|_{L^4})^2 \|\xi\|^4_{L^8}\no \\
&\leq \frac{C_0}{10} \|\nabla \rho\|^2 + C(\|\nabla\bar{\varphi}\|^2_{L^4}+ \|\nabla \varphi^*\|_{L^4}^2) \|\xi\|^4_V ,
\end{align}
\begin{align} \label{147}
|(\F'''(\theta \bar{\varphi} + (1-\theta) \varphi^*)\xi \nabla \xi, \nabla \rho)| & \leq C_F \|\xi\|_{L^4}\|\nabla \xi\|_{L^4}\|\nabla \rho\| \no\\
& \leq C_F \|\xi\|_V \|\xi\|_{H^2} \|\nabla \rho\|\no \\
& \leq \frac{C_0}{10} \|\nabla \rho\|^2 + C \|\xi\|_V ^2\|\xi\|_{H^2}^2,
\end{align}
and using Gagliardo-Nirenberg inequality we get
\begin{align} \label{148}
|(\F'''(\varphi^*) \nabla \varphi^* \rho, \nabla \rho)| &\leq C_F \|\nabla \varphi^*\|_{L^4} \|\rho\|_{L^4} \|\nabla \rho\| \no\\
&\leq C \|\nabla \varphi^*\|_{L^4} (\|\rho\|^{1/2}\|\nabla \rho\|^{1/2}+\|\rho\| ) \|\nabla \rho\| \no\\
&\leq \frac{C_0}{10} \|\nabla \rho\|^2 + C \|\nabla \varphi^*\|_{L^4}^2 \|\rho\|^2 +\frac{C_0}{10} \|\nabla \rho\|^2 + \|\nabla \varphi^*\|^4_{L^4}\|\rho\|^2 \no \\
&\leq \frac{C_0}{5} \|\nabla \rho\|^2 + C(\|\nabla \varphi^*\|_{L^4}^2+\|\nabla \varphi^*\|^4_{L^4})\|\rho\|^2,
\end{align}
Substituting \eqref{144}-\eqref{148} in \eqref{75}, we get
\begin{align}
|&(\Delta (a \rho-\J*\rho +\F'(\bar{\varphi}) - \F'(\varphi^*)-\F''(\varphi^*) \psi), \rho )| + \frac{4C_0}{10} \|\nabla \rho\|^2  \no \\
&\leq C(1 + \|\nabla \varphi^*\|^2_{L^4}+ \|\nabla \varphi^*\|^4_{L^4}) \|\rho\|^2 + C(\|\nabla\bar{\varphi}\|^2_{L^4}+ \|\nabla \varphi^*\|_{L^4}^2) \|\xi\|^4_V + C \|\xi\|_V ^2\|\xi\|_{H^2}^2 \label{84}
\end{align}
Using \eqref{82}, \eqref{83} and \eqref{84} in \eqref{73} we get 
\begin{align}
\frac{1}{2} \frac{\d}{\d t} &\| \rho (t) \|^2 + \frac{3C_0}{10}\|\nabla \rho\|^2 \no \\
\leq &\frac{\nu}{10} \|\nabla \y\|^2 + C(1 +\| \nabla \varphi^*\|^2_{L^4} + \| \nabla \varphi^*\|^4_{L^4} ) \|\rho\|^2 + C \|\nabla \z\|^2 \|\nabla \xi\|^2 \no \\
& + C (\|\nabla\bar{\varphi}\|^2_{L^4}+ \|\nabla \varphi^*\|_{L^4}^2) \|\xi\|^4_V+C \|\xi\|_V ^2\|\xi\|_{H^2}^2 \label{86}
\end{align}
Combining \eqref{85} and \eqref{86}, we get
\begin{align}
\frac{1}{2} \frac{\d}{\d t} \| \rho (t) \|^2 +&\frac{\nu}{10} \|\nabla \y\|^2 + \|\y\|^2 + \frac{3C_0}{10}\|\nabla \rho\|^2 \no  \\
&\leq  C(1 +\| \nabla \varphi^*\|^2_{H^1} + \| \nabla \varphi^*\|^4_{H^1} + \|\varphi^*\|^2_{H^1}) \|\rho\|^2 + C \|\nabla \z\|^2 \|\nabla \xi\|^2 \no \\
& + C (\|\nabla\bar{\varphi}\|^2_{L^4}+ \|\nabla \varphi^*\|_{L^4}^2) \|\xi\|^4_V+C \|\xi\|_V ^2\|\xi\|_{H^2}^2 \no \\
 &+ C \|\xi\|^4_{V} \| \nabla \varphi^*\|^2_{L^4} 
+ C \|\xi\|_V^4 \label{134}
\end{align}
By Gronwall's lemma (differential form) we get
\begin{align*}
\|\rho (t)\|^2 \leq \exp \left(C\int_0^t (1 +\| \nabla \varphi^*\|^2_{L^4} + \| \nabla \varphi^*\|^4_{L^4} + \|\varphi^*\|^2_{V})\d s \right) \left(\int_0^T\alpha(t)\d t \right)
\end{align*} 
where $\alpha(t) = C \|\nabla \z\|^2 \|\nabla \xi\|^2 
 + C (\|\nabla\bar{\varphi}\|^2_{L^4}+ \|\nabla \varphi^*\|_{L^4}^2) \|\xi\|^4_V+C \|\xi\|_V ^2\|\xi\|_{H^2}^2+ C \|\xi\|^4_{V} \| \nabla \varphi^*\|^2_{L^4} 
+ C \|\xi\|_V^4$. \\
Using the bounds \eqref{bound1} and \eqref{bound2} we can show that 
\begin{align*}
\int_0^T \alpha (t) \d t \leq \|\U\|^4_{\L^2(0,T;\V_{\mathrm{div}}')}
\end{align*}
which implies
\begin{align*}
\|\rho (t)\|^2 \leq \|\U\|^4_{\L^2(0,T;\V_{\mathrm{div}}')}
\end{align*}
Integrating \eqref{134} we get
\begin{align*}
\|\rho\|^2_{L^2(0,T;V)} + \|\u \|^2_{L^2(0,T;\V_ \mathrm{div})} \leq C \|\U\|^4_{\L^2(0,T;\V_{\mathrm{div}}')}
\end{align*}
which gives 
\begin{align*}
\|\rho\|^2_{L^\infty(0,T;H)}+\|\rho\|^2_{L^2(0,T;V)} + \|\u \|^2_{L^2(0,T;\V_ \mathrm{div})} \leq C \|\U\|^4_{\L^2(0,T;\V_{\mathrm{div}}')}
\end{align*}
Hence using the inclusions $\mathcal{U} \subset L^\infty(0,T;\G_{\mathrm{div}}) \subset L^2(0,T; \V'_{\mathrm{div}})$, as $\|\U\|_\mathcal{U}\rightarrow 0$
\begin{align*}
\frac{\|S(\U^*+\U)-S(\U^*)-[\w,\psi] \|_\mathcal{V} }{\|\U\|_\mathcal{U}} \leq C\|\U\|_\mathcal{U}\rightarrow 0.
\end{align*}
Hence the proof of the theorem.
\end{proof}

\section{Characterisation of Optimal Control} \label{adjoint}
In this section, we derive the variational inequality satisfied by the optimal control.  Further, we introduce the adjoint system \eqref{ad2}-\eqref{ad18} and discuss its solvability. Finally, we characterize the optimal control in terms of adjoint variables by eliminating the linearised variables $[\psi,\w]$ from \eqref{variational}.  
\subsection{First order optimality condition}
We prove the following theorem using the result of  Theorem \ref{frech}.
\begin{theorem}
Suppose that the hypothesis (\textbf{H1})-(\textbf{H5})  and (\textbf{H7}) are satisfied. Let us assume that $U^* \in \mathcal{U}_{ad}$ is an optimal control for (OCP) such that $S(U^*)=[ \varphi^*,\u^*]$ . Then optimal triplet satisfies 
\begin{align}
\int_0^T \int_\Omega (\varphi^*-\varphi_d) \psi \, dxdt +\int_0^T \int_\Omega (\u^*-\u_d)\cdot \w \, dxdt + \int_\Omega (\varphi^*(T)-\varphi_\Omega) \psi (T) dx \no  \\
+ \int_0^T \int_\Omega \U^* \cdot (\U - \U^*) dx dt \ge 0, \quad \forall \U \in \mathcal{U}_{ad}. \label{variational}
\end{align}
where $[\psi,\w]$ is the unique weak solution of the linearised system \eqref{2}-\eqref{lin1} but replacing $\U$ with $\U-\U^*$ in \eqref{1}.
\end{theorem}
\begin{proof}
Let us denote $\mathrm{G} (\U) = \mathcal{J}(S(\U), \U)$ for all $ \U \in \mathcal{U}_{ad}$. Since $\mathcal{U}_{ad}$ is a convex set, for any minimiser $\U^* \in \mathcal{U}_{ad}$ of $\mathcal{J}$, we have from Lemma 2.21 in \cite{fredi} that 
\begin{align*}
\mathrm{G}'(\U^*) (\U-\U^*) \ge 0, \forall \U \in \mathcal{U}_{ad}.
\end{align*}
where $\mathrm{G}'$ is Fr\'echet derivative. Since $\mathcal{J}$ is in the quadratic functional form, using chain rule we can write the Fr\'echet derivative of $\mathrm{G}$  at every $\U^* \in \mathcal{U}$ as follows 
\begin{align*}
\mathrm{G}'(\U^*) = \mathcal{J}^\prime_{[\varphi^*,\u^*]}(S(\U^*), \U^*) \circ S'(\U^*) + \mathcal{J}'_{\U^*} (S(\U^*), \U^*)
\end{align*}
 and Gateaux derivative of $\mathcal{J}$ in the direction of $( h_1, \h_2$ can be written as
\begin{align*}
\mathcal{J}^\prime_{[\varphi^*,\u^*]}( \varphi^*,\u^*,\U^* )(h_1,\h_2) = \int_0^T \int_\Omega (\varphi^* - \varphi_d)h_1 dxdt+\int_0^T \int_\Omega (\u^* - \u_d) \cdot \h_2 dxdt \\
+ \int_\Omega (\varphi^*(T)-\varphi_\Omega) h_1(T) dx 
\end{align*}
 for all $[h_1,\h_2] \in \mathcal{V}$ and 
\begin{align*}
\mathcal{J}^\prime_{\U^*} (\varphi^*,\u^*, \U^*) (\W)= \int_0^T \int_\Omega \U^* \cdot \W dxdt \quad \forall \W \in \mathcal{U}. 
\end{align*}
Using the fact from Theorem \ref{frech} that $S'(\U^*)(\U-\U^*) = [\psi, \w]$ we get that, 
\begin{align*}
0 \le& ( \mathcal{J}'_{[\varphi^*,\u^*]}(S(\U^*), \U^*) \circ S'(\U^*) + \mathcal{J}'_{\U^*} (S(\U^*), \U^*), \U - \U^*)\\
 =& \int_0^T \int_\Omega (\varphi - \varphi_d)\psi dxdt 
+ \int_0^T \int_\Omega (\u^* - \u_d) \cdot \w dxdt + \int_\Omega (\varphi(T)-\varphi_\Omega) \psi(T) dx \\
&+\int_0^T \int_\Omega \U^* \cdot (\U-\U^*)dxdt.
\end{align*}
Hence follows \eqref{variational}.
\end{proof}
\subsection{Adjoint system}
	
%
Consider the adjoint system corresponding to the optimal control $\U^*$ and corresponding state $(\varphi^*,\u^*)$ 
	\begin{align}
	-\eta_t + \v \cdot \nabla a \varphi^*+ \J*(\v \cdot \nabla \varphi^*) & - (\nabla \J * \varphi^*) \cdot \v - \u^* \cdot \nabla \eta \no \\ 
	-a \Delta \eta + \nabla \J * \nabla \eta - \F''(\varphi^*) \Delta \eta & = \varphi^* - \varphi_d, \label{ad2}\\
	-\nu \Delta \v+\v + \eta \nabla \varphi^* + \nabla q &= \u^*-\u_d,  \label{ad1}\\
	\text{div} \, (\v) & =0, \\
		\v \cdot \n |_{\partial \Omega} = \frac{\partial \eta }{\partial \n}|_{\partial \Omega}& = 0,	\\
	 \eta(T, \cdot) & = \varphi^* (T) - \varphi_\Omega. \label{ad18}
	\end{align}
The weak formulation of the system \eqref{ad1}-\eqref{ad18} is  as follows
\begin{align}
 -_{V'}\langle \eta_t, \chi \rangle_{V} &+ (\v \cdot \nabla a \varphi^*,\chi )+ (\J*(\v \cdot \nabla  \varphi^*),\chi ) - ((\nabla \J  * \varphi^*) \cdot  \v, \chi)  \no\\
- (\u^* \cdot & \nabla \eta, \chi)-(a \Delta \eta,\chi )+ (\nabla \J * \nabla \eta ,\chi)- (\F''(\varphi^*) \Delta \eta, \chi )  = (\varphi^* - \varphi_d, \chi) ,\label{ad20}\\
&\nu(\nabla \v, \nabla \z) +(\v,\z) +(\eta \nabla \varphi^*, \z)= (\u^*-\u_d, \z). \label{ad19}
\end{align}
for every $\chi \in V$ and $\z \in \V_{\mathrm{div}}$ for every $t \in [0,T]$.
\begin{theorem}
Assume that the hypothesis (\textbf{H1})-(\textbf{H5}) and (\textbf{H7}) are satisfied. Then the adjoint system \eqref{ad1}-\eqref{ad18} has a unique weak solution $[\v, \eta]$ satisfying 
\begin{align*}
&\eta \in L^\infty(0,T;H) \cap L^2(0,T;V)\cap H^1(0,T; V'),\\
&\v \in L^2(0,T; \V_{\mathrm{div}}).
\end{align*}
\end{theorem} 
\begin{proof}
We can prove the existence of a weak solution using Faedo-Galerkin approximation as in the proof of Theorem \ref{thmlin}.
We derive the basic estimates that a weak solution should satisfy. Let us take $\chi =\eta$ and $\z=\v$ in \eqref{ad20} and \eqref{ad19} respectively. This leads to
	\begin{align}
		  -\frac{1}{2} \frac{d}{dt}\|\eta\|^2 + (\v \cdot \nabla a \varphi^* , \eta) + (\J*(\v \cdot \nabla \varphi^*),\eta)& - ((\nabla \J * \varphi^*) \cdot \v,\eta) - (\u^* \cdot \nabla \eta, \eta) , \no  \\
	-(a \Delta \eta ,\eta)+ (\nabla \J * \nabla \eta,\eta) &- (\F''(\varphi^*) \Delta \eta ,\eta)   = (\varphi^* - \varphi_d,\eta), \label{ad6} \\
	\nu \|\nabla \v \|^2 + \|\v\|^2 + (\eta \nabla \varphi^*, \v )&= (\u^*-\u_d, \v), \label{ad5}
	\end{align}
Terms in \eqref{ad6} and \eqref{ad5} can be estimated using Gagliardo-Nirenberg and Holder inequalities as follows 
\begin{align}
|(\v \cdot \nabla a \varphi^* , \eta)| &\leq C_J \|\v\|_{L^4} \|\varphi^*\|_{L^4} \|\eta\|  \no \\
& \leq \frac{\nu}{10} \|\nabla \v\|^2 + C \|\varphi^*\|^2_{L^4} \|\eta\|^2, \label{ad7}
\end{align}
\begin{align}
|(\J*(\v \cdot \nabla \varphi^*),\eta)| &\leq C_J \|\v\|_{L^4} \|\nabla \varphi^*\|_{L^4} \|\eta\| \no \\
&\leq \frac{\nu}{10} \|\nabla \v\|^2 + C\|\nabla \varphi^*\|^2_{L^4} \|\eta\|^2 , \label{ad8}
\end{align}
\begin{align}
|((\nabla \J * \varphi^*) \cdot \v,\eta)| &\leq C_J \|\varphi^*\|_{L^4} \|\v\|_{L^4} \|\eta\| \no \\
&\leq \frac{\nu}{10} \|\nabla \v \|^2 + C \|\varphi^*\|^2_{L^4} \|\eta\|^2, \label{ad9}
\end{align}
Using integration by parts and divergence free condition we get
\begin{align}
|(\u \cdot \nabla \eta, \eta)=\int_\Omega \u \left(\frac{\eta^2}{2}\right)' = \int_\Omega \mathrm{div} (\u) \frac{\eta^2}{2}=0, \label{ad10} 
\end{align}
Observe that by integration by parts we get
	\begin{align}
	-(a \Delta \eta ,\eta) - (\F''(\varphi^*) \Delta \eta ,\eta) &= (a + \F''(\varphi^*) \nabla \eta,\nabla \eta) +( \nabla a, \eta \nabla \eta) + (\F'''(\varphi^*) \nabla \varphi^*, \eta \nabla \eta) \no \\
	&\geq C_0 \|\nabla \eta\|^2 + ( \nabla a, \eta \nabla \eta) + (\F'''(\varphi^*) \nabla \varphi^*, \eta \nabla \eta), \label{ad24}
	\end{align}
Right hand side terms of \eqref{ad24} are estimated as follows 	
	\begin{align}
	|(\nabla a, \eta \nabla \eta)| &\leq C_J \|\eta\| \|\nabla \eta\|
	 \leq\frac{C_0}{6} \|\nabla \eta\|^2 + C_J \|\eta\|^2, \label{ad11}
	\end{align}
\begin{align}
	|(\F'''(\varphi^*) \nabla \varphi^*, \eta \nabla \eta)| & \leq C_F \|\nabla \varphi^*\|_{L^4} \|\eta\|_{L^4}\|\nabla \eta \|\no \\
& \leq C_F \|\nabla \varphi^*\|_{L^4}(\|\eta\|^{1/2} \|\nabla \eta\|^{1/2} + \|\eta\| )\|\nabla \eta \|\no \\
& \leq \frac{C_0}{12} \|\nabla \eta\|^2 + \|\nabla \varphi^*\|^4_{L^4} \|\eta\|^2+  \frac{C_0}{12} \|\nabla \eta\|^2 + \|\nabla \varphi^*\|^2_{L^4} \|\eta\|^2 \no\\
& \leq \frac{C_0}{6} \|\nabla \eta\|^2 + (\|\nabla \varphi^*\|^4_{L^4}+\|\nabla \varphi^*\|^2_{L^4})\|\eta\|^2, \label{ad12}
\end{align}
Substituting \eqref{ad11} and \eqref{ad12} in \eqref{ad24} we get 
\begin{align} \label{ad25}
-(a \Delta \eta ,\eta) - (\F''(\varphi^*) \Delta \eta ,\eta) \geq \frac{2C_0}{3} \|\nabla \eta\|^2 -C \|\eta\|^2 - (\|\nabla \varphi^*\|^4_{L^4}+\|\nabla \varphi^*\|^2_{L^4})\|\eta\|^2.
\end{align}
\begin{align}
|(\nabla \J * \nabla \eta,\eta)| &\leq C_J \|\nabla \eta \| \|\eta\| \no\\
& \leq \frac{C_0}{6} \|\nabla \eta\|^2 + C_J \|\eta\|^2, \label{ad13}
\end{align}
\begin{align}
|(\varphi - \varphi_d,\eta)| &\leq \|\varphi - \varphi_d\| \|\eta \| \no \\
&\leq \frac{1}{2} \|\varphi-\varphi_d\|^2 + \frac{1}{2}\|\eta\|^2.
\end{align}
Using \eqref{ad7}, \eqref{ad8}, \eqref{ad9}, \eqref{ad10}, \eqref{ad25}, \eqref{ad13} in \eqref{ad6} we get 
\begin{align}
-\frac{1}{2} \frac{d}{dt} \|\eta\|^2 &\leq \frac{3\nu}{10} \|\nabla \v\|^2  - \frac{C_0}{2} \|\nabla \eta\|^2 + C(1+ \|\varphi^*\|^2_{L^4} + \|\nabla\varphi^*\|^2_{L^4}\no\\
& \quad +\|\nabla\varphi^*\|^4_{L^4}  )\|\eta\|^2 +\frac{1}{2} \|\varphi^*-\varphi_d\|^2. \label{ad15}
\end{align}
For the right hand side of \eqref{ad5} we have the estimates 
\begin{align}
	|(\eta \nabla \varphi^*, \v )| &\leq \|\eta\| \|\nabla \varphi^*\|_{L^4} \|\v\|_{L^4} \no \\
	&\leq \|\eta\| \|\nabla \varphi^*\|_{L^4} \|\nabla \v\| \no \\
	&\leq \frac{\nu}{10} \|\nabla \v\|^2 + \|\eta\|^2 \|\nabla \varphi^*\|^2_{L^4}, \label{ad3}
	\end{align}
	and
\begin{align}
|(\u-\u_d,\v)| &\leq \|\u^* -\u_d\| \|\v\| \no \\
& \leq C_\Omega \|\u^* -\u_d\| \|\nabla \v\| \no \\
& \leq \frac{\nu}{10} \|\nabla \v\|^2 + C \|\u^*-\u_d\|^2. \label{ad4}
\end{align}	
Using \eqref{ad3} and \eqref{ad4} in \eqref{ad5} we get 
\begin{align}
\frac{4\nu}{5} \|\nabla \v \|^2 + \|\v\|^2 \leq C \|\u^*-\u_d\|^2+\|\eta\|^2 \|\nabla \varphi^*\|^2_{L^4}\label{ad14}
\end{align}
Combining \eqref{ad15} and \eqref{ad14} we get 
\begin{align}
&-\frac{1}{2} \frac{d}{dt} \|\eta\|^2 +\frac{\nu}{5} \|\nabla \v\|^2+\|\v\|^2 +\frac{C_0}{2} \|\nabla \eta\|^2 \no \\
&\ \leq C(1+ \|\varphi^*\|^2_{L^4} + \|\nabla\varphi^*\|^2_{L^4}+\|\nabla\varphi^*\|^4_{L^4} )\|\eta\|^2 
+ C (\|\u^*-\u_d\|^2 + \|\varphi^*-\varphi_d\|^2). \label{ad26}
\end{align}
Integrating \eqref{ad26} over $(t,T)$, we get 
\begin{align}
\|\eta(t)\|^2 + \frac{\nu}{5} \int_t^T & \|\nabla \v(s)\|^2 + \int_t^T \|\v(s)\|^2 +\frac{C_0}{2} \int_t^T\|\nabla \eta (s)\|^2 \no \\
&\leq \|\eta(T)\|^2 + C\int_t^T \alpha (s)\|\eta(s)\|^2 ds+C \int_t^T \beta(s) ds, \label{ad17}
\end{align}
where $\alpha(t)=1+ \|\varphi^*\|^2_{L^4} + \|\nabla\varphi^*\|^2_{L^4}+\|\nabla\varphi^*\|^4_{L^4}$ and $\beta(t)=\|(\u^*-\u_d)(t)\|^2 + \|(\varphi^*-\varphi_d)(t) \|^2$.
Using classical Gronwall's inequality 
\begin{align}
\|\eta(t)\|^2 \leq \left[\|\eta(T)\|^2 +C \int_0^T \beta(s) ds\right]\exp \left(C\int_0^T \alpha(s)ds\right). \label{ad16}
\end{align}
Since $\alpha \in L^1(0,T)$ we have that
\begin{align*}
\eta \in L^\infty(0,T; H).
\end{align*}
Using \eqref{ad16} in \eqref{ad17} we get that
\begin{align} \label{ad21}
\v \in L^2(0,T; \V_{\mathrm{div}}), \quad \eta \in L^2(0,T;V). 
\end{align}
In fact, from the estimate
\begin{align*}
\nu \|\nabla \v \|^2 + \|\v\|^2 &=- (\eta \nabla \varphi^*, \v )+(\u^*-\u_d, \v) \\
&\leq C(\|\eta\| \|\nabla \varphi^*\|_{L^4}  + \|\u^* - \u_d\|) \|\nabla \v\|,
\end{align*}
we get 
\begin{align*}
\v \in L^\infty (0,T; \V_{\mathrm{div}}).
\end{align*}
From \eqref{ad20} we also have the following estimate,
\begin{align*}
\| \eta_t\|_{V'} \leq C(\|\varphi^*\| \|\nabla \v\|+\|\nabla \v\| \|\nabla \varphi^*\|+\|\nabla \u^*\| \|\nabla \eta\| +  \|\nabla \eta\| +\| \varphi^* - \varphi_d\| ),
\end{align*}
from which we can deduce, using \eqref{ad21} that 
\begin{align*}
\eta_t \in L^2(0,T;V').
\end{align*}
To prove the uniqueness of the system \eqref{ad2}-\eqref{ad18}, consider two solutions $[\eta_1, \v_1]$ and $[\eta_2, \v_2]$ of the system \eqref{ad2}-\eqref{ad18}. Denoting $\eta=\eta_1-\eta_2, \v=\v_1-\v_2$ and $q=q_1-q_2$, we get
\begin{align}
-\eta_t + \v \cdot \nabla a \varphi^*&+ \J*(\v \cdot \nabla \varphi^*) - (\nabla \J * \varphi^*) \cdot \v - \u^* \cdot \nabla \eta & \no  \\
	-a \Delta \eta + \nabla \J * \nabla \eta - \F''(\varphi^*) \Delta \eta  &= 0, \label{ad27} \\
	-\nu \Delta \v+\v + \eta \nabla \varphi^* + \nabla q &=0, \label{ad28}\\
	\text{div} \, (\v) & =0, \\
		\v \cdot \n |_{\partial \Omega} = \frac{\partial \eta }{\partial \n}|_{\partial \Omega}& = 0,	\\
	 \eta(T, \cdot) & = \varphi^* (T) - \varphi_\Omega.
\end{align}
Taking inner product of \eqref{ad27} and \eqref{ad28} with $\eta$ and $\v$ respectively and recalculating  same estimates above we conclude that the solution to the system \eqref{ad2}-\eqref{ad18} is unique.
\end{proof}
Using the adjoint system \eqref{ad2}-\eqref{ad18}, now we can remove $\psi, \w$ from \eqref{variational}. We have the following lemma.
\begin{lemma}
Suppose (\textbf{H1})-(\textbf{H5}) and (\textbf{H7}) are satisfied. Let $\U^* \in \mathcal{U}_{ad}  $ be an optimal control for (OCP) with corresponding solution $[\varphi^*, \u^*]$ and the solution of adjoint system $[\eta, \v]$. Then we have the following variational inequality.
\begin{align*}
\int_0^T \int_\Omega (\v + \U^*)\cdot(\U - \U^*) \geq 0, \quad \forall \U \in \mathcal{U}_{ad}.
\end{align*}
\end{lemma}
\begin{proof}

Let us take inner product of \eqref{ad2} and \eqref{ad1} with $\psi$ and $ \w$ respectively and add them, where $(\psi,\w)$ is the solution of the linearised system around $( \varphi^*,\u^*)$ with control $\U^*$. We get
\begin{align}
-\langle \eta_t, \psi \rangle + (\v \cdot \nabla a \varphi^*,\psi)+ (\J*(\v \cdot \nabla \varphi^*),\psi) - ((\nabla \J * \varphi^*) \cdot \v,\psi) - (\u^* \cdot \nabla \eta, \psi) \no  \\
	-(a \Delta \eta,\psi ) + (\nabla \J * \nabla \eta ,\psi)- (\F''(\varphi^*) \Delta \eta,\psi)  +
	\nu (\nabla \v, \nabla \w )+(\v, \w) \no \\
	+ (\eta \nabla \varphi^*, \w)  =(\varphi^* - \varphi_d,\psi)+ (\u^*-\u_d,\w) \label{ad22}
\end{align}
Similarly take inner product of \eqref{2} and \eqref{1} with $\eta$ and $ \v$ where $\U$ in the equation for $ \w$ is  replaced with $\U-\U^*$.
\begin{align}
\langle\psi_t, \eta \rangle + (\w \cdot \nabla \varphi^*,\eta) + (\u^* \nabla \psi, \eta)+&\nu (\nabla \w, \nabla \v) +(\w,\v) \no \\
+(\nabla (a \psi - \J * \psi +\F''(\hat{\varphi})\psi), \nabla \eta)& =((a \psi - \J * \psi +\F''(\hat{\varphi})\psi) \nabla \varphi^*, \v)\no \\
 + ((a \varphi - \J * \varphi &+\F'(\varphi^*))\nabla \psi,\v) + (\U-\U^*, \v). \label{ad23}
\end{align}
Subtract \eqref{ad22} from \eqref{ad23} and integrate from $0$ to $T$. Using \eqref{variational} we arrive at 
\begin{align}
\int_0^T \int_\Omega (\v + \U^*)\cdot(\U - \U^*) \geq 0, \quad \forall \U \in \mathcal{U}_{ad}. \label{opcondition}
\end{align}
Since $\mathcal{U}_{ad}$ is a non empty convex closed subset of $\mathcal{U}$, from the first order optimality condition \eqref{opcondition} we can write the optimal control $\U^*$ (see \cite{fredi}), in terms of $\v$, using the projection onto $\mathcal{U}_{ad}$  as
\begin{align*}
\U^* = P_{\mathcal{U}_{ad}} (-\v). 
\end{align*}
Moreover, from the above projection property we can write point wise condition for $\U^*$ as follows
\begin{align*}
\U ^*(x,t) = \max \left\{ \U_1(x,t), \min \left\{ \v(x,t) , \U_2(x,t)  \right\} \right\} 
\end{align*} 
for a.e. $(x,t) \in \Omega \times (0,T)$.

\end{proof}


\bibliographystyle{apa}

\end{document}